\pdfoutput=1
\documentclass{shinyart}
\usepackage[utf8]{inputenc}

\usepackage{shinybib}
\usepackage{relsize}
\usepackage{booktabs}

\newcommand\norm[2]{\left\Vert#1\right\Vert_{#2}}

\newcommand\dual[3]{\left\langle #1, #2\right\rangle_{#3}}

\newcommand\linop[2]{\mathbb L\left[#1,#2\right]}

\renewcommand\C{\mathbb{C}}

\newcommand\N{\mathbb{N}}
\newcommand\R{\mathbb{R}}

\newcommand{\opA}{\mathtt{A}}
\newcommand{\opB}{\mathtt{B}}
\newcommand{\opC}{\mathtt{C}}
\newcommand{\opD}{\mathtt{D}}
\newcommand{\opE}{\mathtt{E}}
\newcommand{\opF}{\mathtt{F}}
\newcommand{\opL}{\mathtt{L}}
\newcommand{\opS}{\mathtt{S}}

\newcommand{\cl}{\operatorname{cl}}

\newcommand{\bd}{\operatorname{bd}}

\newcommand{\supp}{\operatorname{supp}}

\newcommand{\weakly}{\rightharpoonup}

\renewcommand{\subseteq}{\subset}

\newtheorem{assumption}[theorem]{Assumption}

\floatstyle{ruled}
\newfloat{algorithm}{tbp}{loa}
\providecommand{\algorithmname}{Algorithm}
\floatname{algorithm}{\protect\algorithmname}

\usepackage{pgfplots}
\pgfplotsset{compat=newest}

\begin{document}

\title{Optimal control problems with control complementarity constraints}
\subtitle{existence results, optimality conditions, and a penalty method}
\author{%
    Christian Clason%
    \footnote{%
        Universität Duisburg-Essen,
        Faculty of Mathematics,
        45117 Essen,
        Germany,
        \email{christian.clason@uni-due.de},
        \url{https://udue.de/clason}%
    } \and
    Yu Deng%
    \footnote{%
        Technische Universität Bergakademie Freiberg,
        Faculty of Mathematics and Computer Science,
        09596 Freiberg,
        Germany,
        \email{yu.deng@math.tu-freiberg.de},
        \url{http://www.mathe.tu-freiberg.de/nmo/mitarbeiter/yu-deng}%
    }\and
    Patrick Mehlitz%
    \footnote{%
        Brandenburgische Technische Universität Cottbus-Senftenberg,
        Institute of Mathematics, Chair of Optimal Control,
        03046 Cottbus,
        Germany,
        \email{mehlitz@b-tu.de},
        \url{https://www.b-tu.de/fg-optimale-steuerung/team/dr-patrick-mehlitz}%
    }\and
    Uwe Prüfert%
    \footnote{%
        Technische Universität Bergakademie Freiberg,
        Faculty of Mathematics and Computer Science,
        09596 Freiberg,
        Germany,
        \email{uwe.pruefert@math.tu-freiberg.de},
        \url{http://www.mathe.tu-freiberg.de/nmo/mitarbeiter/uwe-pruefert}%
    }
}
\maketitle

\begin{abstract}
    A special class of optimal control problems 
    with complementarity constraints on the control functions is studied. 
    It is shown that such problems possess optimal solutions whenever the underlying control space is a first-order Sobolev space. 
    After deriving necessary optimality conditions of strong stationarity-type, 
    a penalty method based on the Fischer--Burmeister function is suggested and its theoretical properties are analyzed. 
    Finally, the numerical treatment of the problem is discussed and results of computational experiments are presented. 
% 
%     \medskip
% 
%     \noindent\textcolor{structure}{Keywords}\quad    
%     Fischer--Burmeister function, mathematical problems with complementarity constraints, 
%     optimal control, optimality conditions, penalty method
% 
%     \smallskip
% 
%     \noindent\textcolor{structure}{MSC codes}\quad    
%     49J20, 49K20, 49M05, 49M25, 90C48
\end{abstract}

\section{Introduction}

Complementarity conditions appear in many mathematical optimization problems arising from real-world applications, 
and this phenomenon is not restricted to the finite-dimensional setting, 
see \cite{LuoPangRalph1996,Ulbrich2011,Wachsmuth2015} and the references therein. 
A prominent example for a complementarity problem in function spaces is the optimal control of the obstacle problem, 
see \cite{HarderWachsmuth2018} for an overview of existing literature. 
Mathematical problems with complementarity constraints (MPCCs) suffer from an inherent lack of regularity, 
see \cite[Proposition~1.1]{YeZhuZhu1997} and \cite[Lemma~3.1]{MehlitzWachsmuth2016} 
for the finite- and infinite-dimensional situation, respectively, 
which is why the construction of suitable optimality conditions, constraint qualifications, and numerical methods is a challenging task. 
Using so-called \emph{NCP-functions}, complementarity constraints can be transformed 
into possibly nonsmooth equality constraints that can be handled by, e.g., Newton-type methods, 
see \cite{DeLucaFacchineiKanzow2000,FacchineiFischerKanzow1998,Ulbrich2011} and the references therein. 
A satisfying overview of NCP-functions can be found in \cite{SunQi1999}. 
One of the most popular NCP-functions is the so-called \emph{Fischer--Burmeister function} $\phi\colon\R^2\to\R$ given by
\begin{equation}\label{eq:FB_function}
    \forall a,b\in\R\colon\quad \phi(a,b):=\sqrt{a^2+b^2}-a-b,
\end{equation}
see \cite{Fischer1992}. Obviously, one has
\[
    \forall a,b\in\R\colon\quad\phi(a,b)=0\,\Longleftrightarrow\,a\geq 0\,\land\,b\geq 0\,\land ab=0,
\]
which (by definition) holds for all NCP-functions. 
Thus, NCP-functions allow the replacement of a complementarity condition by a single equality constraint. 
In \cite{Ulbrich2011}, it is shown that NCP-functions can be applied to solve complementarity problems 
in function space settings as well.

In this paper, an optimal control problem with complementarity constraints on the control functions is studied. 
Control complementarity constraints have been the subject of several recent papers including \cite{ClarkeDePinho2010,GuoYe2016,MehlitzWachsmuth2016b,PangStewart2008}. 
Classically, such constraints arise from reformulating a bilevel optimal control problem 
with lower level control constraints as a single-level problem using lower level first-order optimality conditions, 
see \cite[Section~5]{MehlitzWachsmuth2016}. 
On the other hand, control complementarity constraints are closely related to switching conditions on the control functions, 
see \cite{ClasonItoKunisch2016,ClasonRundKunischBarnard2016,ClasonRundKunisch2017} and the references therein. 
Here, it will be shown that such problems possess an optimal solution if the control space is taken as $H^1(\Omega)$. 
Recently, optimal control problems with control constraints in first-order Sobolev spaces were studied in \cite{DengMehlitzPruefert2018a,DengMehlitzPruefert2018b}.

It will also be demonstrated that the Fischer--Burmeister function can be used to design penalty methods 
that can be exploited to find minimizers of the corresponding optimal control problem. 
One major advantage of this procedure is that the resulting penalized problems are unconstrained. 
In contrast, simply penalizing the equilibrium condition and leaving the non-negativity conditions 
in the constraints would lead to the appearance of Lagrange multipliers from $H^1(\Omega)^\star$ 
in the necessary optimality conditions of the penalized problems, 
which would cause some theoretical and numerical difficulties due to the presumed high regularity 
of the control space, see \cite{DengMehlitzPruefert2018b}.

The paper is organized as follows: 
In the remainder of this section, the basic notation is introduced. 
Afterwards, the optimal control problem is formally stated and the existence of solutions is discussed in \cref{sec:existence}. 
Necessary optimality conditions of strong stationarity-type are derived in \cref{sec:optimality_conditions}. 
\Cref{sec:num_methods} is dedicated to the theoretical investigation of a penalization procedure. 
The practical implementation of the proposed numerical method 
and some corresponding examples are discussed in \cref{sec:numerics} and  \cref{sec:numerical_examples}, respectively.
A brief summary as well as some concluding remarks are presented in \cref{sec:conclusions}.

\paragraph{Basic notation}
For a real Banach space $\mathcal X$, $\norm{\cdot}{\mathcal X}$ denotes its norm. 
The expression $\mathcal X^\star$ is used to represent the topological dual space of $\mathcal X$. 
Let $\dual{\cdot}{\cdot}{\mathcal X}\colon\mathcal X^\star\times\mathcal X\to\R$
be the associated dual pairing.
For another Banach space $\mathcal Y$, $\linop{\mathcal X}{\mathcal Y}$ represents the 
Banach space of all bounded, linear operators which map from $\mathcal X$ to $\mathcal Y$. 
For $\opF\in\linop{\mathcal X}{\mathcal Y}$, $\opF^\star\in\linop{\mathcal Y^\star}{\mathcal X^\star}$ denotes its adjoint. 
If $\mathcal X\subset\mathcal Y$ holds true 
while the associated identity mapping from $\mathcal X$ into $\mathcal Y$ is continuous, 
then $\mathcal X$ is said to be
continuously embedded into $\mathcal Y$, denoted by $\mathcal X\hookrightarrow\mathcal Y$.

Recall that a set $A\subset\mathcal X$ is said to be weakly sequentially closed 
if all the limit points of weakly convergent sequences contained in $A$ belong to $A$ as well, 
and that any closed, convex set is weakly sequentially closed by Mazur's lemma. 
For any $A\subset\mathcal{X}$, define the polar cone
\begin{align*}
    A^\circ&:=\left\{x^\star\in\mathcal X^\star\,\middle|\,\forall x\in A\colon\,\dual{x^\star}{x}{\mathcal X}\leq 0\right\},
    \shortintertext{as well as the annihilator}
    A^\perp&:=\left\{x^\star\in\mathcal X^\star\,\middle|\,\forall x\in A\colon\,\dual{x^\star}{x}{\mathcal X}=0\right\}.
\end{align*}
By definition, $A^\perp=A^\circ\cap (-A)^\circ$ holds true.
It is well known that $A^\circ$ is a nonempty, closed, convex cone while $A^\perp$ is a closed subspace of $\mathcal X^\star$.
For an arbitrary vector $x\in\mathcal X$, set $x^\perp:=\{x\}^\perp$ for the sake of brevity.

Finally, if a function $F\colon \mathcal X\to\mathcal Y$ is Fr\'{e}chet differentiable at $\bar x\in\mathcal X$, 
then the bounded, linear operator $F'(\bar x)\in\linop{\mathcal X}{\mathcal Y}$ denotes its Fr\'{e}chet derivative at $\bar x$.

\paragraph{Function spaces}

For an arbitrary bounded domain $\Omega\subset\R^d$ and $p\in[1,\infty]$, 
$L^p(\Omega)$ denotes the usual Lebesgue space of (equivalence classes of) Lebesgue measurable functions 
mapping from $\Omega$ to $\R$, which is equipped with the usual norm. 
It is well known that for $p\in[1,\infty)$, the space $L^p(\Omega)^\star$ is isometric to $L^{p'}(\Omega)$ 
for $p'\in(1,\infty]$ such that $1/p+1/p'=1$. The associated dual pairing is given by
\[
    \forall u\in L^p(\Omega)\,\forall v\in L^{p'}(\Omega)\colon\quad \dual{v}{u}{L^p(\Omega)}:=\int_\Omega u(x)v(x)\mathrm d x.
\]
Recall that $L^2(\Omega)$ is a Hilbert space whose dual $L^2(\Omega)^\star$ will be identified with $L^2(\Omega)$ 
by means of Riesz' representation theorem. 
For an arbitrary function $u\in L^1(\Omega)$, $\supp u:=\{x\in\Omega\,|\,u(x)\neq 0\}$ denotes the support of $u$.
Supposing that $A\subset\Omega$ is a Lebesgue measurable set, 
$\chi_A\colon\Omega\to\R$ represents the characteristic function of $A$ which is $1$ for all $x\in A$ and $0$ else. 
Clearly, for a bounded domain $\Omega$ and $p\in[1,\infty)$, the relation $\norm{\chi_A}{L^p(\Omega)}=|A|^{1/p}$ 
is obtained where $|A|$ denotes the Lebesgue measure of $A$.

The Banach space of all weakly differentiable functions from $L^2(\Omega)$ 
whose weak derivatives belong to $L^2(\Omega)$ is denoted by $H^1(\Omega)$.
It is equipped with the usual norm
\[
    \forall y\in H^1(\Omega)\colon
    \quad 
    \norm{y}{H^1(\Omega)}:=\left(\norm{y}{L^2(\Omega)}^2
    +\mathsmaller\sum\nolimits_{i=1}^d\norm{\partial_{x_i}y}{L^2(\Omega)}^2\right)^{1/2}.
\]
Clearly, $H^1(\Omega)$ is a Hilbert space. 
However, its dual $H^1(\Omega)^\star$ will not be identified with $H^1(\Omega)$ so that $H^1(\Omega)$, $L^2(\Omega)$, 
and $H^1(\Omega)^\star$ form a so-called Gelfand triple, i.e., 
they satisfy the relations $H^1(\Omega)\hookrightarrow L^2(\Omega)\hookrightarrow H^1(\Omega)^\star$.
A detailed study of duality in Sobolev spaces can be found in \cite[Section~3]{AdamsFournier2003}.

Whenever $\Omega$ satisfies the so-called cone condition, see \cite[Section~4]{AdamsFournier2003}, 
then the embedding $H^1(\Omega)\hookrightarrow L^2(\Omega)$ is compact, see \cite[Theorem~6.3]{AdamsFournier2003}.
In this paper, $\opE\in\linop{H^1(\Omega)}{L^2(\Omega)}$ is used to denote the latter.

For later use, let $L^2_+(\Omega)\subset L^2(\Omega)$ and $H^1_+(\Omega)\subset H^1(\Omega)$ denote the nonempty, 
closed, and convex cones of almost everywhere nonnegative
functions in $L^2(\Omega)$ and $H^1(\Omega)$, respectively. 

\section{Problem setting and existence of optimal solutions}\label{sec:existence}

In this work, the model 
\textbf{o}ptimal \textbf{c}ontrol problem with \textbf{c}ontrol \textbf{c}omplementarity \textbf{c}onstraints
\begin{equation}\label{eq:OCCC}\tag{OC$^4$} 
    \left\{
        \begin{aligned}       
            \tfrac{1}{2}\norm{\opD[y]-y_\text d}{\mathcal D}^2+J(u,v)&\,\rightarrow\,\min_{y,u,v}\\
            \opA[y]-\opB[u]-\opC[v]&\,=\,0\\
            (u,v)&\,\in\,\C
        \end{aligned}
    \right.
\end{equation}
is studied, where for some $\alpha_1,\alpha_2\geq 0$ and $\varepsilon\geq 0$,
\[
    \forall u,v\in H^1(\Omega)
    \colon\quad 
    J(u,v):=\tfrac{\alpha_1}{2}\norm{u}{L^2(\Omega)}^2+\tfrac{\alpha_2}{2}\norm{v}{L^2(\Omega)}^2
    +\tfrac{\varepsilon}{2}\left(\norm{u}{H^1(\Omega)}^2+\norm{v}{H^1(\Omega)}^2\right),
\]
and $\C$ denotes the complementarity set
\[
    \C:=\left\{(w,z)\in  H^1(\Omega)^2\,\middle|\,0\leq w(x)\perp z(x)\geq 0\text{ a.e. on }\Omega\right\}.
\]
Observing that $\opA$ can represent a differential operator, one can interpret \eqref{eq:OCCC} as
an optimal control problem with complementarity constraints on the control functions 
that can be used to model switching requirements on the controls. 
In the context of ordinary differential equations, optimal control problems 
with mixed control-state complementarity constraints have been studied 
in \cite{ClarkeDePinho2010,GuoYe2016,PangStewart2008} recently.
In \cite{HarderWachsmuth2018,HarderWachsmuth2018b,MehlitzWachsmuth2016b}, 
the interested reader can find some theoretical investigations of optimization problems 
with complementarity constraints with respect to the function spaces $L^2(\Omega)$ and $H^1_0(\Omega)$. 
Recently, optimal control problems with switching constraints related to \eqref{eq:OCCC} 
have been studied in \cite{ClasonRundKunischBarnard2016,ClasonRundKunisch2017}.

For the remainder of this work, the following standing assumptions on the problem \eqref{eq:OCCC} are postulated.
\begin{assumption}\label{ass:OCass}
    The domain $\Omega\subseteq\R^d$ is nonempty, bounded, and satisfies the cone condition.
    Its boundary will be denoted by $\bd\Omega$.
    Let the observation space $\mathcal D$ as well as the state space $\mathcal Y$ be Hilbert spaces.
    The target $y_\textup{d}\in \mathcal D$ will be fixed.
    The operator $\opA\in\linop{\mathcal Y}{\mathcal Y^\star}$ is an isomorphism while 
    $\opB,\opC\in\linop{H^1(\Omega)}{\mathcal Y^\star}$ and $\opD\in\linop{\mathcal Y}{\mathcal D}$ are arbitrarily chosen.
    Finally, $\varepsilon>0$ holds.
\end{assumption}

Let $\opS\in\linop{H^1(\Omega)^2}{\mathcal D}$ be the control-to-observation operator 
which maps any pair of controls $(u,v)\in H^1(\Omega)^2$ to $\opD[y]$, 
where $y\in\mathcal Y$ is the associated uniquely determined solution of the state equation 
\[
    \opA[y]-\opB[u]-\opC[v]=0.
\]
Then, $\opS$ is a well-defined continuous linear operator since $\opA$ is assumed to be an isomorphism. 

\bigskip

In the following, the existence of optimal solutions to \eqref{eq:OCCC} is discussed. 
First, the overall $H^1$-setting needed for the further theoretical treatment of \eqref{eq:OCCC} 
is analyzed in \cref{sec:existenceH1}. 
Some comments on the setting where controls come from $L^2(\Omega)$ are presented in \cref{sec:existenceL2}.

\subsection{First-order Sobolev spaces}\label{sec:existenceH1}

Since the objective function of \eqref{eq:OCCC} is continuously Fr\'{e}chet differentiable, 
convex, and bounded from below, the only critical point for existence is the 
weak sequential closedness of the complementarity set $\C$.  
\begin{lemma}\label{lem:closedness_of_complementarity_set}
    The set $\C$ is closed.
\end{lemma}
\begin{proof}
    Let $\{(u_k,v_k)\}_{k\in\N}\subset\C$ be a sequence converging to $(\bar u,\bar v)\in  H^1(\Omega)^2$. 
    Due to the continuity of the embedding $H^1(\Omega)\hookrightarrow L^2(\Omega)$, 
    the strong convergences $u_k\to\bar u$ and $v_k\to\bar v$ hold in $L^2(\Omega)$. 
    In particular, these convergences hold (at least along a subsequence) pointwise almost
    everywhere. Due to the closedness of the set $\{(a,b)\in\R^2\,|\,0\leq a\perp b\geq 0\}$,
    the desired result follows.
\end{proof}

Although $\C$ is a nonconvex set, the compactness of the embedding $H^1(\Omega)\hookrightarrow L^2(\Omega)$ can be used 
in order to show that $\C$ is weakly sequentially closed.
\begin{lemma}\label{lem:weak_sequential_closedness_of_complementarity_set}
    The set $\C$ is weakly sequentially closed.
\end{lemma}
\begin{proof}
    First, a similar proof as for \cref{lem:closedness_of_complementarity_set} shows that the complementarity set in $L^2(\Omega)$ given by
    \begin{equation}\label{eq:complementarity_set_L2}
        \begin{aligned}[t]
            \widetilde\C :=& 
            \left\{(w,z)\in L^2(\Omega)^2\,\middle|\,0\leq w(x)\perp z(x)\geq 0\text{ a.e. on }\Omega\right\}\\
            =&
            \left\{(w,z)\in L^2_+(\Omega)^2\,\middle|\, \dual{w}{z}{L^2(\Omega)}=0\right\}
        \end{aligned}
    \end{equation}
    is closed as well. 

    Next, choose a sequence $\{(u_k,v_k)\}_{k\in\N}\subset\C$ converging weakly to $(\bar u,\bar v)\in  H^1(\Omega)^2$. 
    Exploiting $u_k\weakly\bar u$ and $v_k\weakly\bar v$ as well as the compactness of the embedding $ H^1(\Omega)\hookrightarrow L^2(\Omega)$, 
    there is a subsequence of $\{(u_k,v_k)\}_{k\in\N}$ that converges strongly to $(\bar u,\bar v)$ in $L^2(\Omega)^2$. 
    Due to the closedness of $\widetilde\C$ in $L^2(\Omega)^2$, $(\bar u,\bar v)\in\widetilde\C\cap H^1(\Omega)^2$ holds, and, consequently, 
    $(\bar u,\bar v)$ is already an element of $\C$. 
    Thus, $\C$ is weakly sequentially closed.
\end{proof}

As a corollary, the existence of optimal solutions to \eqref{eq:OCCC} is obtained.
\begin{corollary}\label{cor:existence}
    The problem \eqref{eq:OCCC} possesses an optimal solution.
\end{corollary}
\begin{proof}
    The objective functional of \eqref{eq:OCCC} is continuously Fr\'{e}chet differentiable, convex, and (due to $\varepsilon>0$) coercive. 
    Furthermore, by \cref{lem:weak_sequential_closedness_of_complementarity_set}, the complementarity set $\C$ is weakly sequentially closed, 
    and so is the feasible set induced by the PDE constraint. 
    Hence, the claim follows by application of Tonelli's direct method.
\end{proof}

\subsection{Lebesgue spaces}\label{sec:existenceL2}

In the remainder of this section, the existence of optimal controls in $L^2(\Omega)$ is investigated.
In this case, the corresponding model problem is given by
\begin{equation}
    \label{eq:problem_l2}
    \tag{OC$_{L^2}$}
    \left\{
        \begin{aligned}
            \tfrac{1}{2}\norm{\opD[y]-y_\text d}{\mathcal D}^2
            +\tfrac{\alpha_1}{2}\norm{u}{L^2(\Omega)}^2
            +\tfrac{\alpha_2}{2}\norm{v}{L^2(\Omega)}^2&\,\rightarrow\,\min_{y,u,v}\\
            \opA[y]-\tilde\opB[u]-\tilde\opC[v]&\,=\,0\\
            (u,v) &\,\in\,\widetilde \C
        \end{aligned}
    \right.
\end{equation}
where the complementarity set $\widetilde\C$ has been defined in \eqref{eq:complementarity_set_L2}.
Furthermore, $\tilde\opB,\tilde\opC\in\linop{L^2(\Omega)}{\mathcal Y^\star}$ need to be chosen.
As already shown in the proof of \cref{lem:weak_sequential_closedness_of_complementarity_set}, $\widetilde\C$ is closed. 
However, $\widetilde \C$ is in general \emph{not} weakly sequentially closed, as the following example shows.
\begin{example}\label{ex:counterex_weakly_sequentially_closed}
    For any $k\in\mathbb N$, define the two open sets
    \[
        \begin{aligned}
            P_k&:=\left\{x\in\R^d\,\middle|\,\mathsmaller\prod\nolimits_{j=1}^d\sin(k\pi x_j)>0\right\},\\
            Q_k&:=\left\{x\in\R^d\,\middle|\,\mathsmaller\prod\nolimits_{j=1}^d\sin(k\pi x_j)<0\right\}.
        \end{aligned}
    \]
    Now, set $u_k:=\chi_{\Omega\cap P_k}$ and $v_k:=\chi_{\Omega\cap Q_k}$. 
    Obviously, $(u_k,v_k)\in\widetilde\C$ holds true for all $k\in\mathbb N$. 
    Furthermore, the sequence $\{(u_k,v_k)\}_{k\in\mathbb N}\subset L^2(\Omega)^2$ converges weakly to the point 
    $(\tfrac{1}{2}\chi_\Omega,\tfrac{1}{2}\chi_\Omega)$, 
    which does not belong to $\widetilde\C$. 
    Thus, $\widetilde\C$ is not weakly sequentially closed.
\end{example}

It may still happen that there exists an optimal solution of the complementarity-constrained problem \eqref{eq:problem_l2}, 
as illustrated by the following example.
For $\mathcal D:=L^2(\Omega)$ and $\mathcal Y:=H^1(\Omega)$, consider the elliptic optimal control problem
\begin{equation}\label{eq:OCCC_easy}
    \left\{
        \begin{aligned}
            \tfrac{1}{2}\norm{\opE[y]-y_\text d}{L^2(\Omega)}^2
            +\tfrac{\alpha_1}{2}\norm{u}{L^2(\Omega)}^2
            +\tfrac{\alpha_2}{2}\norm{v}{L^2(\Omega)}^2&\,\rightarrow\,\min_{y,u,v}&&&\\
            -\nabla \cdot(\mathbf C\nabla y)+\mathbf a y&\,=\,\chi_{\Omega_u}u+\chi_{\Omega_v}v&\quad&\text{a.e. on }\Omega&\\
            \vec{\mathbf n}\cdot(\mathbf C\nabla y)+\mathbf qy&\,=\,0&&\text{a.e. on }\bd\Omega&\\
            (u,v)&\,\in\,\widetilde \C&&&
        \end{aligned}
    \right.
\end{equation}
where we recall that $\opE$ represents the natural embedding from $H^1(\Omega)$ into $L^2(\Omega)$,
$\alpha_1,\alpha_2>0$ are constants, and
$\mathbf C\in L^\infty(\Omega;S^d(\R))$ (where $S^d(\R)$ denotes the set of real symmetric $d\times d$ matrices) 
satisfies the condition of uniform ellipticity, i.e.,
\begin{equation}\label{eq:uniform_ellipticity}
    \exists c_0>0\;\forall x\in \Omega \;\forall \xi \in \R^d\colon
    \quad
    \xi^\top \mathbf{C}(x)\xi\geq c_0\vert\xi\vert_2^2.
\end{equation}
Moreover, $\mathbf a\in L^\infty(\Omega)$ and $\mathbf q\in L^\infty(\bd\Omega)$ are nonnegative and satisfy 
$\norm{\mathbf a}{L^\infty(\Omega)}+\norm{\mathbf q}{L^\infty(\bd\Omega)}>0$, 
and $\Omega_u,\Omega_v\subset\Omega$ are measurable sets of positive measure satisfying $\Omega_u\cup\Omega_v=\Omega$. 
Here, the PDE constraint is interpreted in the weak sense.
It is well known that the associated differential operator $\mathtt A$ is elliptic, see \cite[Section~6]{Evans2010}, 
and, thus, an isomorphism.
\begin{proposition}\label{prop:trivial_existence}
    The problem \eqref{eq:OCCC_easy} possesses an optimal solution.
\end{proposition}
\begin{proof}
    Assume without loss of generality that $\alpha_1\leq \alpha_2$; the other case can be handled analogously. 
    Consider then the surrogate optimal control problem
    \begin{equation}\label{eq:surrogateOC}
        \left\{
            \begin{aligned}
                \tfrac{1}{2}\norm{\opE[y]-y_\text d}{L^2(\Omega)}^2
                +\tfrac{1}{2}\norm{
                    \left(\sqrt{\alpha_1}\chi_{\Omega_u}+\sqrt{\alpha_2}\chi_{\Omega_v\setminus\Omega_u}\right)z
                }{L^2(\Omega)}^2&\,\rightarrow\,\min_{y,z}&&&\\
                -\nabla \cdot(\mathbf C\nabla y)+\mathbf a y&\,=\,z&\quad&\text{a.e. on }\Omega&\\
                \vec{\textbf{n}}\cdot(\mathbf C\nabla y)+\mathbf qy&\,=\,0&&\text{a.e. on }\bd\Omega&\\
                z&\,\in\,L^2_+(\Omega).&&&
            \end{aligned}
        \right.
    \end{equation}
    Note that its objective is equivalent to
    \[
        H^1(\Omega)\times L^2(\Omega)
        \ni(y,z)\mapsto 
        \tfrac{1}{2}\norm{\opE[y]-y_\text d}{L^2(\Omega)}^2
        +\tfrac{\alpha_1}{2}\norm{\chi_{\Omega_u}z}{L^2(\Omega)}^2
        +\tfrac{\alpha_2}{2}\norm{\chi_{\Omega_v\setminus\Omega_u}z}{L^2(\Omega)}^2
        \in\R.
    \]
    The ellipticity of the underlying PDE in \eqref{eq:surrogateOC} implies 
    that the associated control-to-observation operator $\check{\opS}\colon L^2(\Omega)\to L^2(\Omega)$ 
    is linear and continuous, see \cite[Section~6.2]{Evans2010}.
    Observing that $\Omega_u\cup\Omega_v=\Omega$ holds by assumption, the reduced objective functional
    \[
        L^2(\Omega)\ni z
        \mapsto 
        \tfrac{1}{2}\norm{\check{\opS}[z]-y_\text d}{L^2(\Omega)}^2
        +\tfrac{\alpha_1}{2}\norm{\chi_{\Omega_u}z}{L^2(\Omega)}^2
        +\tfrac{\alpha_2}{2}\norm{\chi_{\Omega_v\setminus\Omega_u}z}{L^2(\Omega)}^2\in\R
    \]
    is convex, continuous, and coercive. 
    This shows that the optimal control problem \eqref{eq:surrogateOC} possesses an optimal solution 
    $(\bar y,\bar z)\in H^1(\Omega)\times L^2(\Omega)$ with objective value $\bar m\in\R$.\\
    Let $(y,u,v)\in H^1(\Omega)\times L^2(\Omega)\times L^2(\Omega)$ be feasible to \eqref{eq:OCCC_easy}. 
    Defining $z:=\chi_{\Omega_u}u+\chi_{\Omega_v\setminus \Omega_u}v$, $(y,z)$ is feasible for \eqref{eq:surrogateOC}. 
    Then, the estimate
    \begin{multline*}
        \tfrac{1}{2}\norm{\opE[y]-y_\text d}{L^2(\Omega)}^2
        +\tfrac{\alpha_1}{2}\norm{u}{L^2(\Omega)}^2
        +\tfrac{\alpha_2}{2}\norm{v}{L^2(\Omega)}^2\\
        \begin{aligned}
            &\geq\tfrac{1}{2}\norm{\opE[y]-y_\text{d}}{L^2(\Omega)}^2
            +\tfrac{\alpha_1}{2}\norm{\chi_{\Omega_u}u}{L^2(\Omega)}^2
            +\tfrac{\alpha_2}{2}\norm{\chi_{\Omega_v\setminus\Omega_u}v}{L^2(\Omega)}^2\\
            &=\tfrac{1}{2}\norm{\opE[y]-y_\text{d}}{L^2(\Omega)}^2
            +\tfrac{\alpha_1}{2}\norm{\chi_{\Omega_u}z}{L^2(\Omega)}^2
            +\tfrac{\alpha_2}{2}\norm{\chi_{\Omega_v\setminus\Omega_u}z}{L^2(\Omega)}^2 \geq\bar m
        \end{aligned}
    \end{multline*}
    is obtained.
    In particular, the objective value of \eqref{eq:OCCC_easy} is bounded from below by $\bar m$.

    Define $\bar u:=\chi_{\Omega_u}\bar z$ and $\bar v:=\chi_{\Omega_v\setminus\Omega_u}\bar z$. 
    Then, $(\bar y,\bar u,\bar v)$ is feasible to \eqref{eq:OCCC_easy} 
    since $\bar y$ is the state associated with $\bar z$ and $\chi_{\Omega_u}\bar u+\chi_{\Omega_v}\bar v=\bar z$ holds true. 
    Moreover, the relation
    \begin{multline*}
        \tfrac{1}{2}\norm{\opE[\bar y]-y_\text d}{L^2(\Omega)}^2  
        +\tfrac{\alpha_1}{2}\norm{\bar u}{L^2(\Omega)}^2
        +\tfrac{\alpha_2}{2}\norm{\bar v}{L^2(\Omega)}^2\\
        =\tfrac{1}{2}\norm{\opE[\bar y]-y_\text d}{L^2(\Omega)}^2
        +\tfrac{\alpha_1}{2}\norm{\chi_{\Omega_u}\bar z}{L^2(\Omega)}^2
        +\tfrac{\alpha_2}{2}\norm{\chi_{\Omega_v\setminus\Omega_u}\bar z}{L^2(\Omega)}^2=\bar m
    \end{multline*}
    follows.
    Thus, $(\bar y,\bar u,\bar v)$ is an optimal solution of \eqref{eq:OCCC_easy}.
\end{proof}
Note that the proof of \cref{prop:trivial_existence} yields a strategy for the solution of \eqref{eq:OCCC_easy} 
by means of standard arguments from optimal control by solving the surrogate problem \eqref{eq:surrogateOC}. 

\section{Optimality conditions}\label{sec:optimality_conditions}

Consider the so-called state-\emph{reduced} problem
\begin{equation}\label{eq:OCCC_reduced}
    \left\{
        \begin{aligned}
            \tfrac{1}{2}\norm{\opS [u,v]-y_\text{d}}{\mathcal D}^2+J(u,v)&\,\rightarrow\,\min_{u,v}\\
            (u,v)&\,\in\,\C
        \end{aligned}
    \right.
\end{equation}
which is equivalent to \eqref{eq:OCCC} by definition of the control-to-observation operator $\opS $. 
Using the embedding operator $\opE\colon H^1(\Omega)\to L^2(\Omega)$, \eqref{eq:OCCC_reduced} can be stated equivalently as
\begin{equation}\label{eq:OCCC_as_MPCC}
    \left\{
        \begin{aligned}
            \tfrac{1}{2}\norm{\opS [u,v]-y_\text{d}}{\mathcal D}^2+J(u,v)&\,\rightarrow\,\min_{u,v}\\
            \opE[u]&\,\in\, L^2_+(\Omega)\\
            \opE[v]&\,\in\, L^2_+(\Omega)\\
            \dual{\opE[u]}{\opE[v]}{L^2(\Omega)}&\,=\,0
        \end{aligned}
    \right.
\end{equation}
which is a generalized MPCC in the Banach space $L^2(\Omega)$. 
It was shown in \cite[Lemma~3.1]{MehlitzWachsmuth2016} that 
Robinson's constraint qualification,
see \cite[Section~2.3.4]{BonnansShapiro2000} for its definition, some discussion, and 
suitable references to the literature, does not hold at the feasible points of this problem. 
Moreover, since $\opE$ is not surjective, the constraint qualifications needed to show 
that locally optimal solutions of this problem satisfy MPCC-tailored stationarity conditions 
(e.g., the weak or strong stationarity conditions) are not satisfied, 
see \cite{MehlitzWachsmuth2016,Wachsmuth2015} for details.

On the other hand, it is still possible to derive necessary optimality conditions for \eqref{eq:OCCC_reduced} 
using a standard trick from finite-dimensional MPCC theory: 
Define appropriate surrogate problems which do not contain a complementarity constraint anymore 
and handle them with the classical KKT conditions in Banach spaces.

In order to formulate an appropriate surrogate problem, let $(\bar u,\bar v)\in H^1(\Omega)^2$ 
be a feasible point of \eqref{eq:OCCC_reduced} and  
define the measurable sets
\begin{align}
    I^{+0}(\bar u,\bar v)&:=\{x\in\Omega\,|\,\bar u(x)>0\,\land\,\bar v(x)=0\},
    \label{eq:I+0}\\
    I^{0+}(\bar u,\bar v)&:=\{x\in\Omega\,|\,\bar u(x)=0\,\land\,\bar v(x)>0\},
    \label{eq:I0+}\\
    I^{00}(\bar u,\bar v)&:=\{x\in\Omega\,|\,\bar u(x)=0\,\land\,\bar v(x)=0\}.
    \label{eq:I00}
\end{align}
Noting that $L^2(\Omega)$ is a space of equivalence classes, it should be mentioned
that these sets are well-defined up to sets of Lebesgue measure zero.
This will be taken into account in the following.
If $(\bar u,\bar v)$ is a locally optimal solution of \eqref{eq:OCCC_reduced}, 
then it is also a locally optimal solution of the auxiliary problems
\begin{equation}\label{eq:rNLP1}\tag{rNLP$_{\bar u}$}
    \left\{
        \begin{aligned}
            \tfrac{1}{2}\norm{\opS [u,v]-y_\text{d}}{\mathcal D}^2+J(u,v)&\,\rightarrow\,\min_{u,v}&&&\\
            u&\,\geq\,0&\quad&\text{a.e. on }I^{+0}(\bar u,\bar v)&\\
            u&\,=\,0&\quad&\text{a.e. on }I^{0+}(\bar u,\bar v)\cup I^{00}(\bar u,\bar v)&\\
            v&\,\geq 0&&\text{a.e. on }I^{0+}(\bar u,\bar v)\cup I^{00}(\bar u,\bar v)&\\
            v&\,=\,0&&\text{a.e. on }I^{+0}(\bar u,\bar v)&
        \end{aligned}
    \right.
\end{equation}
and
\begin{equation}\label{eq:rNLP2}\tag{rNLP$_{\bar v}$}
    \left\{
        \begin{aligned}
            \tfrac{1}{2}\norm{\opS [u,v]-y_\text{d}}{\mathcal D}^2+J(u,v)&\,\rightarrow\,\min_{u,v}&&&\\
            u&\,\geq\,0&\quad&\text{a.e. on }I^{+0}(\bar u,\bar v)\cup I^{00}(\bar u,\bar v)&\\
            u&\,=\,0&\quad&\text{a.e. on }I^{0+}(\bar u,\bar v)&\\
            v&\,\geq 0&&\text{a.e. on }I^{0+}(\bar u,\bar v)&\\
            v&\,=\,0&&\text{a.e. on }I^{+0}(\bar u,\bar v)\cup I^{00}(\bar u,\bar v)
        \end{aligned}
    \right.
\end{equation}
since their respective feasible sets are smaller than $\mathbb C$ but contain $(\bar u,\bar v)$. 
By standard notion, see \cite{PangFukushima1999,ScheelScholtes2000,Wachsmuth2015}, 
\eqref{eq:rNLP1} and \eqref{eq:rNLP2} are referred to as \emph{restricted nonlinear problems}. 
Furthermore, the corresponding \emph{relaxed nonlinear problem} is introduced by means of
\begin{equation}\label{eq:RNLP}\tag{RNLP}
    \left\{
        \begin{aligned}
            \tfrac{1}{2}\norm{\opS [u,v]-y_\text{d}}{\mathcal D}^2+J(u,v)&\,\rightarrow\,\min_{u,v}&&&\\
            u&\,\geq\,0&\quad&\text{a.e. on }I^{+0}(\bar u,\bar v)\cup I^{00}(\bar u,\bar v)&\\
            u&\,=\,0&\quad&\text{a.e. on }I^{0+}(\bar u,\bar v)&\\
            v&\,\geq 0&&\text{a.e. on }I^{0+}(\bar u,\bar v)\cup I^{00}(\bar u,\bar v) &\\
            v&\,=\,0&&\text{a.e. on }I^{+0}(\bar u,\bar v).
        \end{aligned}
    \right.
\end{equation}
Observe that the feasible points $(u,v)\in H^1(\Omega)^2$ of \eqref{eq:RNLP} 
do not necessarily satisfy the complementarity condition $(u,v)\in\C$. 
Combining standard techniques from finite-dimensional MPCC theory and optimization 
in Banach spaces, the following result is obtained, 
see also \cite[Theorems~3.1 and 5.2]{Wachsmuth2015}.
It should be noted that due to the appearance of the two control variables
$u$ and $v$ in \eqref{eq:OCCC_reduced}, there will be two Lagrange multipliers $\mu$ and $\nu$ corresponding
to $u$ and $v$, respectively, in the stationarity system as well. 
In particular, the pair $(\mu,\nu)\in H^1(\Omega)^\star\times H^1(\Omega)^\star$ may be identified with a functional
from $(H^1(\Omega)^2)^\star$.
\begin{theorem}\label{thm:S_stationarity}
    Let $(\bar u,\bar v)\in H^1(\Omega)^2$ be a locally optimal solution of \eqref{eq:OCCC_reduced}. 
    Then, there exist multipliers $\mu,\nu\in H^1(\Omega)^\star$ satisfying
    \begin{subequations}\label{eq:SSt}
        \begin{align}
            \label{eq:SSt1}
            &0=\opS ^\star\bigl[\opS [\bar u,\bar v]-y_\textup{d}\bigr]+J'(\bar u,\bar v)+(\mu,\nu),\\
            \label{eq:SSt2}
            &\mu\in
            \left\{
                z\in H^1(\Omega)\,\middle|
                \begin{aligned}
                    &z\geq 0&&\text{a.e. on }I^{+0}(\bar u,\bar v)\cup I^{00}(\bar u,\bar v)\\
                    &z=0&&\text{a.e. on }I^{0+}(\bar u,\bar v)
                \end{aligned}
            \right\}^\circ,\\
            \label{eq:SSt3}
            &\dual{\mu}{\bar u}{ H^1(\Omega)}=0,\\
            \label{eq:SSt4}
            &\nu\in
            \left\{
                z\in H^1(\Omega)\,\middle|
                \begin{aligned}
                    &z\geq 0&&\text{a.e. on }I^{0+}(\bar u,\bar v)\cup I^{00}(\bar u,\bar v)\\
                    &z=0&&\text{a.e. on }I^{+0}(\bar u,\bar v)
                \end{aligned}
            \right\}^\circ,\\
            \label{eq:SSt5}
            &\dual{\nu}{\bar v}{ H^1(\Omega)}=0.
        \end{align}
    \end{subequations}
\end{theorem} 
\begin{proof}
    Introducing the cones
    \begin{align*}
        \mathcal K_{+0}&:=\left\{z\in H^1(\Omega)\,\middle|
            \begin{aligned}
                &z\geq 0&&\text{a.e. on }I^{+0}(\bar u,\bar v)\\
                &z=0 &&\text{a.e. on }I^{0+}(\bar u,\bar v)\cup I^{00}(\bar u,\bar v)
        \end{aligned}\right\},\\
        \mathcal K_{0+,00}&:=\left\{z\in H^1(\Omega)\,\middle|
            \begin{aligned}
                &z\geq 0&&\text{a.e. on }I^{0+}(\bar u,\bar v)\cup I^{00}(\bar u,\bar v)\\
                &z=0&&\text{a.e. on }I^{+0}(\bar u,\bar v)
        \end{aligned}\right\},
    \end{align*}
    \eqref{eq:rNLP1} is equivalent to
    \[
        \left\{
            \begin{aligned}
                \tfrac{1}{2}\norm{\opS [u,v]-y_\text{d}}{\mathcal D}^2+J(u,v)&\,\rightarrow\,\min_{u,v}\\
                u&\,\in\,\mathcal K_{+0}\\
                v&\,\in\,\mathcal K_{0+,00}.
            \end{aligned}
        \right.
    \]
    Since $(\bar u,\bar v)$ is a locally optimal solution of \eqref{eq:rNLP1}, 
    there exist multipliers $\mu^1,\nu^1\in H^1(\Omega)^\star$ which satisfy the corresponding KKT conditions
    \begin{equation}\label{eq:KKTrNLP1}
        \left\{
            \begin{aligned}
                &0=\opS ^\star\bigl[\opS [\bar u,\bar v]-y_\text{d}\bigr]+J'(\bar u,\bar v)+(\mu^1,\nu^1),\\
                &\mu^1\in\mathcal K_{+0}^\circ\cap\bar u^\perp,\\
                &\nu^1\in\mathcal K_{0+,00}^\circ\cap\bar v^\perp,
            \end{aligned}
        \right.
    \end{equation}
    see \cite[Theorem~3.9]{BonnansShapiro2000}. 
    Considering \eqref{eq:rNLP2} in a similar way, there exist $\mu^2,\nu^2\in H^1(\Omega)^\star$ which satisfy
    \begin{equation}\label{eq:KKTrNLP2}
        \left\{
            \begin{aligned}
                &0=\opS ^\star\bigl[\opS [\bar u,\bar v]-y_\text{d}\bigr]+J'(\bar u,\bar v)+(\mu^2,\nu^2),\\
                &\mu^2\in\mathcal K_{+0,00}^\circ\cap\bar u^\perp,\\
                &\nu^2\in\mathcal K_{0+}^\circ\cap\bar v^\perp,
            \end{aligned}
        \right.
    \end{equation}
    where 
    \begin{align*}
        \mathcal K_{+0,00}&:=\left\{z\in H^1(\Omega)\,\middle|
            \begin{aligned}
                &z\geq 0&&\text{a.e. on }I^{+0}(\bar u,\bar v)\cup I^{00}(\bar u,\bar v)\\
                &z=0 &&\text{a.e. on }I^{0+}(\bar u,\bar v)
        \end{aligned}\right\},\\
        \mathcal K_{0+}&:=\left\{z\in H^1(\Omega)\,\middle|
            \begin{aligned}
                &z\geq 0&&\text{a.e. on }I^{0+}(\bar u,\bar v)\\
                &z=0&&\text{a.e. on }I^{+0}(\bar u,\bar v)\cup I^{00}(\bar u,\bar v)
        \end{aligned}\right\}.
    \end{align*}
    Combining the respective first condition in \eqref{eq:KKTrNLP1} and \eqref{eq:KKTrNLP2} 
    yields $\mu^1=\mu^2$ and $\nu^1=\nu^2$. 
    Since $\mathcal K_{+0,00}^\circ\cap\bar u^\perp$ is a subset of $\mathcal K_{+0}^\circ\cap\bar u^\perp$ 
    while $\mathcal K_{0+,00}^\circ\cap\bar v^\perp$ is a subset of $\mathcal K_{0+}^\circ\cap\bar v^\perp$, 
    the desired result is obtained by setting $\mu:=\mu^2$ and $\nu:=\nu^1$.  
\end{proof}
Note that the system \eqref{eq:SSt} coincides with the KKT conditions of \eqref{eq:RNLP}. 
In this regard, it is reasonable to call the conditions \eqref{eq:SSt} a \emph{strong stationarity-type system}. 

\begin{remark}\label{rem:multipliers}
    It is difficult to give an explicit characterization of the multipliers $\mu,\nu\in H^1(\Omega)^\star$. 
    Assume that $\Omega$ has a Lipschitz boundary.
    Introducing $\mathcal H_\mathcal A:=\{z\in H^1(\Omega)\,|\,z=0\;\text{a.e. on }\mathcal A\}$ for a fixed measurable
    set $\mathcal A\subset \Omega$ and using the relation 
    $H^1_+(\Omega)^\circ=H^1(\Omega)^\star\cap\mathcal M_-(\overline\Omega)$, 
    see \cite[Lemma~3.1]{DengMehlitzPruefert2018b}, it holds that
    \[
        \mu     \in\left(H^1_+(\Omega)\cap\mathcal H_{I^{0+}(\bar u,\bar v)}\right)^\circ
        =\cl\left(H^1(\Omega)^\star\cap\mathcal M_-(\overline\Omega)+\mathcal H_{I^{0+}(\bar u,\bar v)}^\perp\right)
    \]
    where $\mathcal M_-(\overline\Omega)$ denotes the set of all finite, nonpositive Borel measures on $\overline\Omega$.
    A similar result can be obtained to characterize $\nu$. However, due to the appearance of the closure as well
    as the annihilated subspace associated with $\mathcal H_{I^{+0}(\bar u,\bar v)}$, this characterization is of limited practical use;
    in particular, it cannot be deduced that $\mu$ and $\nu$ are measures.
    Applying the machinery of capacity theory, see \cite{AttouchButtazzoMichaille2006,BonnansShapiro2000},
    a more advanced approach to the characterization of $\mu$ and $\nu$ can be attempted.
    For this purpose, one could strengthen the constraints in \eqref{eq:rNLP1}, \eqref{eq:rNLP2}, and \eqref{eq:RNLP} to hold 
    \emph{quasi-everywhere} on the respective subdomains, i.e., the respective conditions hold up to sets of $H^1$-capacity zero. 
    Then, one needs to find explicit expressions for the polar cone associated with sets of type
    \[
        \left\{z\in H^1(\Omega)\,\middle|\,
            \begin{aligned}
                z\geq 0&\quad\text{quasi-everywhere on }\mathcal A\\
                z=0 &\quad\text{quasi-everywhere on }\Omega\setminus \mathcal A
            \end{aligned}
        \right\}
    \]
    where $\mathcal A\subset\Omega$ is measurable. 
    The price one has to pay when using this approach is a less restrictive stationarity system than \eqref{eq:SSt}. 
    In particular, the polar cones from \eqref{eq:SSt2} and \eqref{eq:SSt4} would be replaced by larger ones.
\end{remark}

In order to state necessary optimality conditions of strong stationarity-type that avoid 
the appearance of multipliers and allow a numerical implementation, one 
can exploit the definition of the polar cone in the system \eqref{eq:SSt}. 
\begin{corollary}\label{cor:consequences_of_S_Stationarity}
    Let $(\bar u,\bar v)\in H^1(\Omega)^2$ be a locally optimal solution of \eqref{eq:OCCC_reduced}. 
    Then, the condition
    \[
        0=\dual{\opS [\bar u,\bar v]-y_\textup{d}}{\opS [\bar u,\bar v]}{\mathcal D}+J'(\bar u,\bar v)[\bar u,\bar v]
    \]
    holds, and for any pair $(z_u,z_v)\in H^1_+(\Omega)\times  H^1_+(\Omega)$, 
    \[
        \left.
            \begin{aligned}
                &\supp z_u\subseteq I^{+0}(\bar u,\bar v)\cup I^{00}(\bar u,\bar v)\\
                &\supp z_v\subseteq I^{0+}(\bar u,\bar v)\cup I^{00}(\bar u,\bar v)
            \end{aligned}
        \right\}
        \,\Longrightarrow\,
        \dual{\opS ^\star\bigl[\opS [\bar u,\bar v]-y_\textup{d}\bigr]
        +J'(\bar u,\bar v)}{(z_u,z_v)}{ H^1(\Omega)^2}\geq 0.
    \]
\end{corollary}
\begin{proof}
    Due to \cref{thm:S_stationarity}, there exist $\mu,\nu\in H^1(\Omega)^\star$ satisfying \eqref{eq:SSt}. 
    Testing \eqref{eq:SSt1} with $(\bar u,\bar v)$ while exploiting \eqref{eq:SSt3}, \eqref{eq:SSt5}, 
    and the definition of the adjoint operator, 
    the first statement of the corollary follows. 

    The second statement is a consequence of \eqref{eq:SSt1}, \eqref{eq:SSt2}, and \eqref{eq:SSt4}.
\end{proof}

\begin{remark}\label{rem:MPCC_S_Stationarity}
    According to standard terminology for MPCCs, the necessary optimality conditions \eqref{eq:SSt} are of strong stationarity-type, 
    see, e.g., \cite[Definition~5.1]{Wachsmuth2015} and \cite[Definition~2.7]{Ye2005}. 
    Recall that a feasible point $(\bar u,\bar v)\in  H^1(\Omega)^2$ of \eqref{eq:OCCC_reduced} and thus
    of \eqref{eq:OCCC_as_MPCC} is a strongly stationary point of \eqref{eq:OCCC_as_MPCC} 
    in the sense of \cite[Definition~5.1]{Wachsmuth2015} if and only if there are multipliers $(\mu,\nu)\in L^2(\Omega)^2$ satisfying
    \begin{subequations}\label{eq:classical_S_stationarity}
        \begin{align}
            \label{eq:classical_S_stationarity_1}
            &0=\opS ^\star\bigl[\opS [\bar u,\bar v]-y_\textup d\bigr]+J'(\bar u,\bar v)+(\opE,\opE)^\star[\mu,\nu],\\
            \label{eq:classical_S_stationarity_2}
            &\mu=0\quad \text{a.e. on }I^{+0}(\bar u,\bar v),\\
            \label{eq:classical_S_stationarity_3}
            &\nu=0\quad \text{a.e. on }I^{0+}(\bar u,\bar v),\\
            \label{eq:classical_S_stationarity_4}
            &\mu\leq 0\,\land\,\nu\leq 0\quad \text{a.e. on }I^{00}(\bar u,\bar v),
        \end{align}
    \end{subequations}
    see also \cite[Definition~4.1]{MehlitzWachsmuth2016b}.
    If $\C$ is replaced by $\widetilde \C$ and $\varepsilon =0$ is taken in the definition of $J$
    (in which case $\opE$ is the identity mapping), 
    the systems \eqref{eq:SSt} and \eqref{eq:classical_S_stationarity} are equivalent. 
    However, for $\C$ and $\varepsilon >0$, the necessary optimality conditions \eqref{eq:SSt} are weaker than \eqref{eq:classical_S_stationarity}, 
    which can be seen as follows:
    It is clear that whenever $(\tilde\mu,\tilde\nu)\in L^2(\Omega)^2$ 
    satisfy the classical strong stationarity conditions \eqref{eq:classical_S_stationarity}, 
    then the multipliers $\mu:=\opE^\star[\tilde\mu]$ and $\nu:=\opE^\star[\tilde\nu]$ satisfy \eqref{eq:SSt}. 
    On the other hand, by means of \cref{thm:S_stationarity}, the multipliers appearing in the system \eqref{eq:SSt} 
    may come from $H^1(\Omega)^\star\setminus L^2(\Omega)$ in general. 
\end{remark}

\begin{remark}\label{rem:applicability}
    In this section, only the property of $\opS$ to be a bounded, linear operator has been exploited. 
    Thus, the optimality conditions obtained in \cref{thm:S_stationarity} and \cref{cor:consequences_of_S_Stationarity}
    are applicable in many different situations, e.g., in case where $\opS$ is the control-to-observation operator associated
    with a linear elliptic equation where $u$ and $v$ only operate on some subdomain,
    or for a linear parabolic equation where the controls $u$ and $v$ only depend on time. The latter problems are closely related
    to the switching-constrained problems examined in \cite{ClasonItoKunisch2016,ClasonRundKunischBarnard2016,ClasonRundKunisch2017}.

    It should be noted that similar necessary optimality conditions can be derived if $\opS\colon H^1(\Omega)^2\to\mathcal D$
    is Fr\'{e}chet differentiable but not necessarily linear.
\end{remark}

\section{Penalization of complementarity constraints}\label{sec:num_methods}

In order to find optimal solutions of \eqref{eq:OCCC}, an obvious idea would be to penalize the violation of the equilibrium condition
\begin{equation}\label{eq:equilibrium_condition}
    u(x)v(x)\,=\,0\quad\text{a.e. on }\Omega
\end{equation}
in \eqref{eq:OCCC}. 
This is related to the approaches used in 
\cite{ClasonItoKunisch2016,ClasonRundKunischBarnard2016,ClasonRundKunisch2017}
for the treatment of switching-constrained optimal control problems. 
However, the resulting penalized problem would still involve inequality constraints for the controls in $H^1(\Omega)$, 
and thus the associated KKT conditions would involve Lagrange multipliers 
from $H^1(\Omega)^\star\cap\mathcal M_-(\overline\Omega)$, 
see \cite[Section~5]{DengMehlitzPruefert2018b} for details. 
This, however, may provoke theoretical and numerical difficulties that should be avoided here.

To get around these issues, 
the penalization of the overall complementarity constraint using the Fischer--Burmeister function 
is proposed here, which leads to penalized problems in which the only constraint is the state equation.

\subsection{Penalty term}\label{sec:penalty_terms}

Let $\phi:\R^2\to\R$ denote the Fischer--Burmeister function introduced in \eqref{eq:FB_function} and let 
the mapping $\Phi\colon L^2(\Omega)^2\to L^2(\Omega)$ 
be the associated Nemytskii operator defined by
\[
    \forall (w,z)\in L^2(\Omega)^2\,\forall x\in\Omega\colon\quad \Phi(w,z)(x):=\phi(w(x),z(x)).
\]
This operator is well-defined since for all $w,z\in L^2(\Omega)$, one has
\[
    \begin{aligned}
        \norm{\Phi(w,z)}{L^2(\Omega)}
    &\leq \left(\int_\Omega\bigr(w^2(x)+z^2(x)\bigr)\mathrm d x\right)^{1/2}+\norm{w}{L^2(\Omega)}+\norm{z}{L^2(\Omega)}\\
&\leq \left(\int_\Omega\bigr(|w(x)|+|z(x)|\bigr)^2\mathrm dx\right)^{1/2}+\norm{w}{L^2(\Omega)}+\norm{z}{L^2(\Omega)}\\
&\leq 2\left(\norm{w}{L^2(\Omega)}+\norm{z}{L^2(\Omega)}\right)<+\infty,
    \end{aligned}
\]
i.e., $\Phi$ maps from $L^2(\Omega)^2$ to $L^2(\Omega)$, see also \cite[Section~3.3]{Ulbrich2011}.
For a detailed introduction to the theory of superposition operators in Lebesgue spaces, the
interested reader is referred to \cite{AppellZabrejko1990,GoldbergKampowskyTroeltzsch1992}.

The violation of the complementarity constraint $(u,v)\in\C$ can then be penalized using the functional $F\colon H^1(\Omega)^2\to\R^+_0$ 
defined by
\begin{equation}\label{eq:penalty}
    \forall (u,v)\in H^1(\Omega)^2\colon\quad   
    F(u,v):=\tfrac{1}{2}\int_\Omega\phi^2(u(x),v(x))\mathrm dx=\tfrac{1}{2}\norm{\Phi(\opE[u],\opE[v])}{L^2(\Omega)}^2.
\end{equation}
Recall that $\opE\in\linop{H^1(\Omega)}{L^2(\Omega)}$ represents the natural embedding $H^1(\Omega)\hookrightarrow L^2(\Omega)$.

It is obvious that $\Phi$ cannot be Fr\'{e}chet differentiable since $\phi$ is not smooth. 
In contrast, $F$ is a continuously Fr\'{e}chet differentiable mapping.
\begin{lemma}\label{lem:subdifferential_L2_penalization}
    Let $(\bar u,\bar v)\in H^1(\Omega)^2$ be arbitrarily chosen. 
    Then, $F$ is continuously Fr\'{e}chet differentiable at $(\bar u,\bar v)$. 
    The associated Fr\'{e}chet derivative is given by
    \[
        \forall(\delta^u,\delta^v)\in H^1(\Omega)^2\colon
        \quad
        F'(\bar u,\bar v)[\delta^u,\delta^v]
        =
        \int_\Omega\phi(\bar u(x),\bar v(x))\bigl(\eta_{\bar u}(x)\delta^u(x)+\eta_{\bar v}(x)\delta^v(x)\bigr)\mathrm dx,
    \]
    where $\eta_{\bar u},\eta_{\bar v}\in L^\infty(\Omega)$ are defined by
    \begin{subequations}\label{eq:char_subdiff_8}
        \begin{align}
            \forall x\in\Omega\colon\quad     \eta_{\bar u}(x)                                        & =
            \begin{cases}
                \tfrac{\bar u(x)}{\sqrt{\bar u(x)^2+\bar v(x)^2}}-1 & \text{if }x\notin I^{00}(\bar u,\bar v),\\
                0                                                   & \text{if }x\in I^{00}(\bar u,\bar v),
            \end{cases}\\
            \forall x\in\Omega\colon\quad \eta_{\bar v}(x)                                        & =
            \begin{cases}
                \tfrac{\bar v(x)}{\sqrt{\bar u(x)^2+\bar v(x)^2}}-1 & \text{if }x\notin I^{00}(\bar u,\bar v),\\
                0                                                   & \text{if }x\in I^{00}(\bar u,\bar v),
            \end{cases}
        \end{align}
    \end{subequations}
    and $I^{00}(\bar u,\bar v)$ is defined by \eqref{eq:I00}.
\end{lemma}
\begin{proof}
    Let $f\colon\R^2\to\R$ be given by
    \[
        \forall (a,b)\in\R^2\colon\quad f(a,b):=\tfrac{1}{2}\phi(a,b)^2.
    \]
    One can check that $f$ is continuously differentiable with gradient
    \[
        \forall (a,b)\in\R^2\colon\quad 
        \nabla f(a,b)=
        \begin{cases}
            \phi(a,b)\begin{pmatrix}\tfrac{a}{\sqrt{a^2+b^2}}-1\\\tfrac{b}{\sqrt{a^2+b^2}}-1\end{pmatrix}       &       \text{if }(a,b)\neq(0,0),\\
            \begin{pmatrix}0\\0\end{pmatrix}                                                                                                                            &       \text{if }(a,b)=(0,0).
        \end{cases}
    \]
    Clearly, the Nemytskii-operator $\mathcal F$ associated with $f$ maps from $L^2(\Omega)^2$ to $L^1(\Omega)$,
    since $\Phi$ maps $L^2(\Omega)^2$ to $L^2(\Omega)$.
    Noting that $a/\sqrt{a^2+b^2}\in[-1,1]$ and $b/\sqrt{a^2+b^2}\in[-1,1]$ hold for all $(a,b)\in\R^2\setminus\{(0,0)\}$, 
    the Nemytskii operator associated with $\nabla f$ maps from $L^2(\Omega)^2$ to $L^2(\Omega)^2$. 
    Applying \cite[Theorems~4 and 7]{GoldbergKampowskyTroeltzsch1992}, 
    $\mathcal F\colon L^2(\Omega)^2\to L^1(\Omega)$ is continuously Fr\'{e}chet differentiable. Furthermore,
    \[
        \forall x\in\Omega\colon
        \quad 
        \mathcal F'(w,z)[\delta_w,\delta_z](x)
        =
        \nabla_a f(w(x),z(x))\delta_w(x)+\nabla_bf(w(x),z(x))\delta_w(x)
    \]
    for any $(w,z),(\delta_w,\delta_z)\in L^2(\Omega)^2$.

    Define $\opL\in\linop{L^1(\Omega)}{\R}$ by $\opL[w]:=\int_\Omega w(x)\mathrm dx$.
    Then, $F=\opL\circ\mathcal F\circ(\opE,\opE)$. 
    Since all involved mappings are continuously Fr\'{e}chet differentiable, 
    the assertion of the lemma follows by exploiting the chain rule for Fr\'{e}chet differentiable functions, 
    see \cite[Theorem~2.20]{Troeltzsch2009}.
\end{proof}

\begin{remark}\label{rem:other_penalty_terms}
    As the penalty functional $F$ is smooth, 
    it cannot lead to exact penalization of the complementarity constraints,
    see, e.g., \cite[Theorem~5.9]{GeigerKanzow2002}. 
    Although \cref{sec:numerical_examples} demonstrates that a penalty method using $F$ 
    behaves well in numerical practice, in principle any other NCP-function, see \cite{SunQi1999} for an overview, 
    can be used to construct similar penalty methods.

    One possible alternative would be to use $F_1\colon H^1(\Omega)^2\to\R^+_0$ given by
    \[
        \forall (u,v)\in H^1(\Omega)^2\colon
        \quad 
        F_1(u,v):=\int_\Omega|\phi(u(x),v(x))|\mathrm dx
        =\norm{\tilde\Phi(\opE[u],\opE[v])}{L^1(\Omega)},
    \]
    where $\tilde\Phi\colon L^2(\Omega)^2\to L^1(\Omega)$ is the mapping $\opE_{L^2\to L^1}\circ\Phi$
    where $\opE_{L^2\to L^1}$ represents the continuous embedding $L^2(\Omega)\hookrightarrow L^1(\Omega)$. 
    This leads to a nonsmooth but Lipschitz continuous mapping. 

    Another approach would be to exploit the so-called \emph{smoothed} Fischer--Burmeister function
    $\phi_\theta\colon\R^2\to\R$ given by
    \[
        \forall (a,b)\in\R^2\colon
        \quad
        \phi_\theta(a,b):=\sqrt{a^2+b^2+2\theta}-a-b,
    \]
    which is continuously differentiable for any $\theta>0$, see \cite{Kanzow1996}.
    Using \cite[Theorems~4 and~7]{GoldbergKampowskyTroeltzsch1992}, one can check that the 
    associated Nemytskii operator $\tilde \Phi_\theta\colon L^2(\Omega)^2\to L^1(\Omega)$ is 
    continuously Fr\'{e}chet differentiable. 
    Define $F_{1,\theta}\colon H^1(\Omega)^2\to\R^+_0$ by means of
    \[
        \forall (u,v)\in H^1(\Omega)^2\colon
        \quad 
        F_{1,\theta}(u,v):=\int_\Omega|\phi_\theta(u(x),v(x))|\mathrm dx
        =\norm{\tilde \Phi_\theta(\opE[u],\opE[v])}{L^1(\Omega)}.
    \]
    Clearly, $F_{1,0}$ corresponds to $F_1$. For $\theta>0$ this approach
    can be seen as a mixture of a penalty and a smoothing method. 
    However, it needs to be noted that $F_{1,\theta}$
    is nonsmooth even for positive values of $\theta$.
\end{remark}

\subsection{Existence, convergence results, and optimality conditions}\label{sec:convergence_results}

Using the penalty functional $F$ defined in \eqref{eq:penalty} to penalize 
the complementarity constraints in \eqref{eq:OCCC_reduced} leads to the family of penalized problems
\begin{equation}\label{eq:penalized_problem}\tag{P$_k$}
    \tfrac{1}{2}\norm{\opS [u,v]-y_\text{d}}{\mathcal D}^2+J(u,v)+\sigma_k F(u,v)\,\rightarrow\,\min_{u,v},
\end{equation}
where $\{\sigma_k\}_{k\in\N}\subset\R^+$ is a sequence of positive real numbers tending to infinity as $k\to\infty$.
The first question is about the existence of solutions of \eqref{eq:penalized_problem}.
\begin{proposition}\label{prop:existence_P_sigma}
    For any $\sigma_k>0$, the penalized problem \eqref{eq:penalized_problem} possesses an optimal solution.
\end{proposition}
\begin{proof}
    Let $\{(u_l,v_l)\}_{l\in\N}\subset H^1(\Omega)^2$ be a minimizing sequence for \eqref{eq:penalized_problem} 
    and let $\bar m\in\overline{\R}$ be the corresponding infimal value. 
    Since $J$ is, due to $\varepsilon>0$, coercive and bounded from below, this sequence is bounded in $H^1(\Omega)^2$ and, thus, 
    possesses a weakly convergent subsequence (without relabeling) with weak limit $(\bar u,\bar v)\in H^1(\Omega)^2$. 
    Due to the compactness of $H^1(\Omega)\hookrightarrow L^2(\Omega)$, the strong convergences $u_l\to\bar u$ and $v_l\to\bar v$ hold in $L^2(\Omega)$. 
    Noting that the operator $\Phi$ is continuous on $L^2(\Omega)^2$, see \cite[Theorem~4]{GoldbergKampowskyTroeltzsch1992}, it follows that
    \[
        \lim_{l\rightarrow\infty}F(u_l,v_l)=F(\bar u,\bar v).
    \]
    Thus, the continuity of $\opS $ and the weak lower semicontinuity of norms imply that
    \begin{multline*}
        \tfrac{1}{2}\norm{\opS [\bar u,\bar v]-y_\text{d}}{\mathcal D}^2+J(\bar u,\bar v)+\sigma_kF(\bar u,\bar v)\\
        \begin{aligned}
            &\leq \liminf_{l\rightarrow\infty}\left(\tfrac{1}{2}\norm{\opS [u_l,v_l]-y_\text{d}}{\mathcal D}^2+J(u_l,v_l)\right)
            +\sigma_k\lim_{l\rightarrow\infty}F(u_l,v_l)\\
            &=\liminf_{l\rightarrow\infty}\left(\tfrac{1}{2}\norm{\opS [u_l,v_l]-y_\text{d}}{\mathcal D}^2+J(u_l,v_l) +\sigma_kF(u_l,v_l)\right)=\bar m,
        \end{aligned}
    \end{multline*}
    i.e., $(\bar u,\bar v)$ is a global minimizer of \eqref{eq:penalized_problem}. 
\end{proof}

Next, the convergence of solutions of \eqref{eq:penalized_problem} as $\sigma_k\to\infty$ is addressed.
\begin{proposition}\label{prop:strong_convergence_of_surrogate_solutions}
    Fix a sequence $\{\sigma_k\}_{k\in\N}\subset\R^+$ tending to infinity as $k\to\infty$. 
    For any $k\in\N$, let $(u_k,v_k)\in H^1(\Omega)^2$ be a global minimizer of \eqref{eq:penalized_problem}. 
    Then, $\{(u_k,v_k)\}_{k\in\N}$ contains a subsequence converging strongly in $H^1(\Omega)^2$ 
    to a point $(\bar u,\bar v)\in\C$ such that $(\bar y,\bar u,\bar v)$, 
    where $\bar y\in\mathcal Y$ is the state associated with $(\bar u,\bar v)$, 
    is an optimal solution of \eqref{eq:OCCC}.

    Moreover, any subsequence of $\{(u_k,v_k)\}_{k\in\N}$ converging weakly to some 
    $(\bar u,\bar v)$ in $ H^1(\Omega)^2$ 
    produces a global minimizer of \eqref{eq:OCCC} in the above sense.
\end{proposition}
\begin{proof}
    For any $k\in\N$, the estimate
    \[
        \tfrac{1}{2}\norm{\opS [u_k,v_k]-y_\text{d}}{\mathcal D}^2+J(u_k,v_k)
        +\sigma_kF(u_k,v_k)\leq \tfrac{1}{2}\norm{y_\text{d}}{\mathcal D}^2
    \]
    follows from the feasibility of $(0,0)\in H^1(\Omega)^2$ for \eqref{eq:penalized_problem}. 
    Thus, since $J$ is coercive and bounded from below while $F$ only takes nonnegative values, 
    $\{(u_k,v_k)\}_{k\in\N}$ is bounded and therefore 
    contains a weakly convergent subsequence (which, as all further subsequences, will not be relabeled). 
    Recalling the compactness of $H^1(\Omega)\hookrightarrow L^2(\Omega)$, 
    the sequence $\{(u_k,v_k)\}_{k\in\N}$ converges strongly to $(\bar u,\bar v)$ 
    in $L^2(\Omega)^2$ and thus pointwise almost everywhere at least along a subsequence. 
    Furthermore, the relation
    \[
        0\leq \norm{\Phi(\opE[u_k],\opE[v_k])}{L^2(\Omega)}\leq \sqrt{\tfrac{1}{\sigma_k}}\norm{y_\text{d}}{\mathcal D}\to 0
    \]
    is obtained 
    as $k\rightarrow\infty$. Consequently, at least along a subsequence, $\{\Phi(\opE[u_k],\opE[v_k])\}_{k\in\N}$ converges pointwise a.e. to $0$. 
    By definition of $\Phi$, $(\bar u,\bar v)\in\C$ follows.

    Now choose $(u,v)\in\C$ arbitrarily. Since this point is feasible to \eqref{eq:penalized_problem}, it follows for any $k\in\N$ that
    \[
        \begin{aligned}
            \tfrac{1}{2}\norm{\opS [u,v]-y_\text{d}}{\mathcal D}^2+J(u,v)
            & \geq \tfrac{1}{2}\norm{\opS [u_k,v_k]-y_\text{d}}{\mathcal D}^2+J(u_k,v_k)+\sigma_kF(u_k,v_k)\\
            &\geq \tfrac{1}{2}\norm{\opS [u_k,v_k]-y_\text{d}}{\mathcal D}^2+J(u_k,v_k).
        \end{aligned}
    \]
    Thus, using the weak lower semicontinuity of the functionals, one obtains
    \begin{equation*}
        \begin{aligned}
            \tfrac{1}{2}\norm{\opS [\bar u,\bar v]-y_\text d}{\mathcal D}^2+J(\bar u,\bar v)
            &\leq\liminf_{k\to\infty}\left(\tfrac{1}{2}\norm{\opS [u_k,v_k]-y_\text d}{\mathcal D}^2+J(u_k,v_k)\right)\\
            &\leq\limsup_{k\to\infty}\left(\tfrac{1}{2}\norm{\opS [u_k,v_k]-y_\text d}{\mathcal D}^2+J(u_k,v_k)\right)\\
            &\leq \limsup_{k\to\infty}\left(\tfrac{1}{2}\norm{\opS [u_k,v_k]-y_\text d}{\mathcal D}^2 +J(u_k,v_k)+\sigma_kF(u_k,v_k)\right)\\
            &\leq\tfrac{1}{2}\norm{\opS [u,v]-y_\text d}{\mathcal D}^2+J(u,v)         
        \end{aligned}
    \end{equation*}
    for all $(u,v)\in\C$. 
    Consequently, $(\bar u,\bar v)$ is a global minimizer of the state-reduced problem \eqref{eq:OCCC_reduced}.
    Choosing $u:=\bar u$ and $v:=\bar v$ in the above estimate, one obtains
    \[
        \tfrac{1}{2}\norm{\opS [u_k,v_k]-y_\text d}{\mathcal D}^2+J(u_k,v_k)
        \to\tfrac{1}{2}\norm{\opS [\bar u,\bar v]-y_\text d}{\mathcal D}^2+J(\bar u,\bar v),
    \]
    and $J(u_k,v_k)\to J(\bar u,\bar v)$ follows by \cref{lem:sum_of_real_sequences}.
    Since $u_k\to\bar u$ and $v_k\to\bar v$ in $L^2(\Omega)$,
    the definition of $J$ and $\varepsilon>0$ imply that
    \[
        \norm{u_k}{H^1(\Omega)}^2+\norm{v_k}{H^1(\Omega)}^2
        \to \norm{\bar u}{H^1(\Omega)}^2+\norm{\bar v}{H^1(\Omega)}^2.
    \]
    Now, applying \cref{lem:sum_of_real_sequences} once more yields
    \[
        \norm{u_k}{H^1(\Omega)}^2\to\norm{\bar u}{H^1(\Omega)}^2,\qquad
        \norm{v_k}{H^1(\Omega)}^2\to\norm{\bar v}{H^1(\Omega)}^2.
    \]
    Combining this with the weak convergences $u_k\weakly\bar u$ and $v_k\weakly\bar v$ in $H^1(\Omega)$, 
    the convergences $u_k\to\bar u$ and $v_k\to\bar v$ in $H^1(\Omega)$ follow since the latter is a Hilbert space. 
    This yields the first assertion.

    If $\{(u_k,v_k)\}_{k\in\N}$ contains a subsequence converging weakly to 
    some $(\bar u,\bar v)\in H^1(\Omega)^2$ in $H^1(\Omega)^2$, then the above arguments can be partially repeated to show that $(\bar u,\bar v)$ is a global minimizer of \eqref{eq:OCCC_reduced}.
    This completes the proof.
\end{proof}

An obvious advantage of \eqref{eq:penalized_problem} is that it is a smooth and unconstrained problem, 
allowing the straightforward derivation of necessary optimality conditions.
Hence, the following result is a direct consequence of Fermat's rule and 
\cref{lem:subdifferential_L2_penalization}.
\begin{proposition}\label{prop:NOC_surrogate_L2_penalization}
    For fixed $\sigma_k>0$, let $(u_k,v_k)\in H^1(\Omega)^2$ be a locally optimal solution of \eqref{eq:penalized_problem}. 
    Then, the corresponding functions $\eta_{u_k},\eta_{v_k}\in L^\infty(\Omega)$ defined as in \eqref{eq:char_subdiff_8} satisfy
    \[
        0=\opS ^\star\bigl[\opS [u_k,v_k]-y_\textup{d}\bigr]
        +J'(u_k,v_k)
        \sigma_k(\opE,\opE)^\star[\Phi(\opE[u_k],\opE[v_k])\eta_{u_k},\Phi(\opE[u_k],\opE[v_k])\eta_{v_k}].
    \]
\end{proposition}

\begin{remark}
    Similar results as in this section can be shown for the penalty terms induced 
    by the nonsmooth functionals $F_1$ and $F_{1,\theta_k}$ given in \cref{rem:other_penalty_terms} 
    using the continuity of the associated Nemytskii operators $\tilde \Phi$ and $\tilde \Phi_{\theta_k}$ 
    as well as calculus rules for Clarke's generalized derivative, see \cite{Clarke:1990a}.
    Obtaining a convergence result as in \cref{prop:strong_convergence_of_surrogate_solutions} for $F_{1,\theta_k}$ 
    additionally requires to choose $\sigma_k$ and $\theta_k$ such that $\sigma_k\sqrt{\theta_k}\to 0$ as $k\to\infty$.
\end{remark}
\begin{remark}
    Using the boundedness of the solutions and passing to subsequences, it is possible by pointwise inspection 
    to take the limit $k\to\infty$ in the optimality system from \cref{prop:NOC_surrogate_L2_penalization}
    and derive the existence of multipliers $\mu,\nu\in H^1(\Omega)^\star$ which satisfy
    the polarity relations from \cref{thm:S_stationarity} with respect to the index sets $I^{+0}(\bar u,\bar v)$
    and $I^{0+}(\bar u,\bar v)$. This can be seen as a natural extension of the so-called weak stationarity concept,
    see \cite[Definition~4.1]{MehlitzWachsmuth2016b}, to
    \eqref{eq:OCCC_reduced}. However, it does not seem to be possible to infer the polarity relations for
    $\mu$ and $\nu$ on $I^{00}(\bar u,\bar v)$ found in the strong stationarity system from
    \cref{thm:S_stationarity}. Noting that our penalty approach is related to Scholtes' relaxation
    technique for the numerical solution of finite-dimensional MPCCs which yields so-called Clarke-stationary
    points in general, see \cite[Section~3.1]{HoheiselKanzowSchwartz2013} for details, this observation 
    does not seem to be too surprising since Clarke-stationarity is much weaker than strong stationarity.
\end{remark}

\section{Numerical treatment}\label{sec:numerics}

This section deals with the numerical implementation of the penalization technique described in \cref{sec:num_methods} 
following a ``first-discretize-then-optimize approach'' based on a finite element discretization.
In order to concentrate on the complementarity constraint, the state equation is chosen as the elliptic model problem
\begin{equation}\label{eq:PDE}\tag{PDE}
    \left\{
        \begin{aligned}       
            -\nabla\cdot(\mathbf{C}\nabla y)+\mathbf{a}y&\,=\,\mathbf bu+\mathbf cv && \text{a.e. on } \Omega\\
            \vec{\mathbf{n}}\cdot(\mathbf{C}\nabla y)&\,=\,0&&\text{a.e. on }\bd \Omega.       
        \end{aligned}
    \right.
\end{equation}
Here, $\Omega\subset \R^d$ is a domain with Lipschitz boundary $\bd(\Omega)$, 
$\mathbf{C}\in L^\infty(\Omega;S^d(\R))$ satisfies 
the condition of uniform ellipticity \eqref{eq:uniform_ellipticity}, and 
the functions $\mathbf{a}, \mathbf{b}, \mathbf{c}\in L^\infty(\Omega)$
do not vanish while $\mathbf a$ is additionally nonnegative, see also \cref{sec:existenceL2}.
Set $\mathcal{D}:=L^2(\Omega)$ and $\mathcal{Y}:=H^1(\Omega)$. 
The operator $\opD:=\opE$ represents the natural embedding $H^1(\Omega)\hookrightarrow L^2(\Omega)$.
Note that the weak formulation of the associated state equation can be written in the abstract form 
$\opA[y]-\opB[u]-\opC[v]=0$, where the bounded, linear operators $\opA,\opB,\opC \in \mathbb{L}[H^1(\Omega), H^1(\Omega)^\star] $ 
are given for all $y,u,v,w\in H^1(\Omega)$ as 
\begin{align*}
    \dual{\opA[y]}{w}{H^1(\Omega)}&:=\int_{\Omega}(\mathbf{C}(x)\nabla y(x))\cdot \nabla w(x)\mathrm{d}x+\int_{\Omega}\mathbf{a}(x)y(x)w(x)\mathrm{d}x, \\ 
    \dual{\opB[u]}{w}{H^1(\Omega)}&:=\int_{\Omega}\mathbf{b}(x)u(x) w(x)\mathrm{d}x,\\
    \dual{\opC[v]}{w}{H^1(\Omega)}&:=\int_{\Omega}\mathbf{c}(x)v(x) w(x)\mathrm{d}x.
\end{align*}
It can be checked that the operator $\opA$ is elliptic and self-adjoint under the postulated assumptions, see, e.g., \cite[Section~6]{Evans2010}.
The operators $\opB$ and $\opC$ are self-adjoint as well.

\subsection{Finite element discretization}\label{sec:discrete_NOC}

While the discretization of \eqref{eq:penalized_problem} is rather standard, 
some notation needs to be introduced for the sake of the following subsection.
Let the domain $\Omega$ be discretized by a suitable tessellation $\Omega_\Delta$, 
where $n_p$ denotes the number of vertices and $n_e$ the number of elements in $\Omega_\Delta$. 
All functions from  $H^1(\Omega)$ ($y$,  $u$, $v$, and  $p$) are represented by finite elements 
from $\mathcal{P}^1(\Omega_\Delta)$. 
The corresponding coefficient vectors are denoted by $\vec{y}$, $\vec{u}$, $\vec{v}$, and $\vec{p}$, respectively. 
The set of test functions $H^1(\Omega)$ is represented by the same basis functions.

The coefficient functions $\mathbf{C}$, $\mathbf{a}$, $\mathbf{b}$, and $\mathbf{c}$ 
as well as the desired state $y_\text d$ are assumed to be chosen from $L^\infty(\Omega)$ and 
discretized by functions from $\mathcal{P}^0(\Omega_\Delta)$; their discrete approximations 
are denoted by $C$, $\vec{a}$, $\vec{b}$, $\vec{c}$, and $\vec{y}_\text{d}$, respectively.
The matrix $E_{10}\in\R^{n_e\times n_p}$ realizes the discrete projection of $\mathcal{P}^1$ approximations into $\mathcal P^0$
and corresponds to the natural embedding operator $\opE\colon H^1(\Omega)\to L^2(\Omega)$. 
The mass matrices $M_0(1)$ and $M_1(1)$ correspond to the finite element spaces $\mathcal{P}^0(\Omega_\Delta)$ 
and $\mathcal{P}^1(\Omega_\Delta)$, respectively. 
The stiffness matrix associated with the constant coefficient $1$ (i.e., $\mathbf{C}$ is the identity in $\R^{d\times d}$) is denoted by $K(1)$. 
A detailed description of this discretization and the specific forms of these matrices can be found in \cite{DengMehlitzPruefert2018a}. 

The main difficulty when discretizing \eqref{eq:penalized_problem}
lies in the handling of the penalty term $F(u,v)$. 
Since the Fischer--Burmeister function is penalized with respect to the space $L_2(\Omega)$, 
the mass matrix $M_0(1)$ can be used to evaluate integrals over all elements. 
Interpreting powers and square roots of a vector in a componentwise fashion,
a reasonable discretization of $F(u,v)$ is given by
\[
    \tilde F(\vec{u},\vec{v})
    =
    \tfrac{1}{2} \left(\sqrt{(E_{10}\vec{u})^2+(E_{10}\vec{v})^2}-E_{10}\vec{u}-E_{10}\vec{v}\right)^{\top}
    M_0(1)
    \left(\sqrt{(E_{10}\vec{u})^2+(E_{10}\vec{v})^2}-E_{10}\vec{u}-E_{10}\vec{v}\right)
\]
for all $\vec u,\vec v\in\R^{n_p}$.
The appearance of $\opE_{10}$ is motivated by the proof of \cref{lem:subdifferential_L2_penalization}, 
where the penalty functional $F$ has been represented as the composition of three differentiable mappings:
The natural embedding $\opE:H^1(\Omega)\to L^2(\Omega)$, the Nemytskii-operator associated 
with the squared Fischer--Burmeister function
(as a mapping from $L^2(\Omega)^2$ to $L^1(\Omega)$), and a linear integral operator. 
This discretization strategy leads to the finite-dimensional problem associated with \eqref{eq:penalized_problem} given by
\begin{equation}\label{eq:discretized_penalized_problem}
    \left\{
        \begin{aligned}
            \tfrac{1}{2}(E_{10}\vec{y}-\vec{y}_\text{d})^{\top}M_0(1)(E_{10}\vec{y}-\vec{y}_\text{d})
            +\tfrac{\alpha_1}{2}\vec{u}^{\top}M_1(1)\vec{u}+\tfrac{\alpha_2}{2}\vec{v}^{\top}M_1(1)\vec{v}\qquad&\\
            +\tfrac{\varepsilon}{2}\vec{u}^{\top}(M_1(1)+K(1))\vec{u}
            +\tfrac{\varepsilon}{2}\vec{v}^{\top}(M_1(1)+K(1))\vec{v}+\sigma_k\tilde F(\vec{u},\vec{v}) &\,\rightarrow\,\min_{\vec y,\vec u,\vec v}\\
            (M_{1}(\vec a)+K(C))\vec{y}-M_{1}(\vec b)\vec{u}-M_{1}(\vec c)\vec{v}&\,=\,0.
        \end{aligned}
    \right.
\end{equation}

For the optimality conditions, one first observes that the quadratic function $\tilde F$ 
is differentiable everywhere and that its derivative at $(\vec{u},\vec{v})$ is given by
\begin{equation}\label{eq:FB_First_Order_deriv}
    \tilde F'(\vec{u},\vec{v})=
    \begin{pmatrix}
        E_{10}^{\top}\text{diag}\left(T_u(\vec{u},\vec{v})\right) 
        M_0(1)\left(\sqrt{(E_{10}\vec{u})^2+(E_{10}\vec{v})^2}-E_{10}\vec{u}-E_{10}\vec{v}\right)\\
        E_{10}^{\top}\text{diag}\left(T_v(\vec{u},\vec{v})\right) 
    M_0(1)\left(\sqrt{(E_{10}\vec{u})^2+(E_{10}\vec{v})^2}-E_{10}\vec{u}-E_{10}\vec{v}\right)                                                 \end{pmatrix},
\end{equation}
where the vectors $T_u(\vec{u},\vec{v}),T_v(\vec{u},\vec{v})\in\R^{n_e}$ are defined for all $i\in\{1,\ldots,n_e\}$ as
\[      
    \begin{aligned}
        T_u(\vec{u},\vec v)_i&:=
        \begin{cases}
            \frac{(E_{10}\vec u)_i}{\sqrt{(E_{10}\vec{u})_i^2+(E_{10}\vec{v})_i^2}}-1   
            &\text{if }(\opE_{10}\vec u)_i\neq 0\text{ or }(\opE_{10}\vec v)_i\neq 0,\\
            0   &\text{if }(\opE_{10}\vec u)_i=(\opE_{10}\vec v)_i=0,
        \end{cases}     \\
        T_v(\vec{u},\vec v)_i&:=
        \begin{cases}
            \frac{(E_{10}\vec v)_i}{\sqrt{(E_{10}\vec{u})_i^2+(E_{10}\vec{v})_i^2}}-1   
            &\text{if }(\opE_{10}\vec u)_i\neq 0\text{ or }(\opE_{10}\vec v)_i\neq 0,\\
            0   &\text{if }(\opE_{10}\vec u)_i=(\opE_{10}\vec v)_i=0.
        \end{cases}     \\      
    \end{aligned}
\]
Note that the case $(\opE_{10}\vec u)_i=(\opE_{10}\vec v)_i=0$ corresponds to the \emph{biactive} case, i.e., where the discretized controls
$\vec u$ and $\vec v$ (interpreted in the discretized counterpart of $L^2(\Omega)$, i.e., elementwise) are zero at the same time.

Combining \eqref{eq:discretized_penalized_problem} and \eqref{eq:FB_First_Order_deriv}, 
it is now possible to obtain the following KKT system for the problem \eqref{eq:discretized_penalized_problem}:
\begin{subequations}\label{eq:NOC_discrete}
    \begin{align}
        E_{10}^{\top}M_0(1)E_{10}\vec{y}-E_{10}^{\top}M_{0}(1)\vec{y}_\text{d}-(M_{1}(\vec{a})+K(C))\vec{p}&=0\\
        \left[\alpha_1M_1(1)+\varepsilon\left(M_1(1)+K(1)\right)\right]\vec{u}
        +\sigma_k \tilde F'_{\vec u}(\vec{u}, \vec{v}) +M_{1}(\vec{b})\vec{p}&=0\\
        \left[\alpha_2M_1(1)+\varepsilon\left(M_1(1)+K(1)\right)\right]\vec{v}
        +\sigma_k \tilde F'_{\vec v}(\vec{u}, \vec{v}) +M_{1}(\vec{c})\vec{p}&=0\\
        -(M_{1}(\vec{a})+K(C))\vec{y}+M_{1}(\vec{b})\vec{u}+M_{1}(\vec{c})\vec{v}&=0.
    \end{align}
\end{subequations}  
Recall that $\vec{p}$ represents the discretized adjoint state and can also be considered as a multiplier related to the discretized state equation. 
Since the function $\tilde F'$ is nonsmooth but Lipschitz continuous, 
the nonlinear system \eqref{eq:NOC_discrete} can be solved using a damped semismooth Newton-type method, see \cite{QiSun1999}.
Note that the domain of nonsmoothness associated with the mapping 
$\tilde F'\colon\R^{n_p}\times\R^{n_p}\to\R^{n_p}\times\R^{n_p}$ is given by 
\[
    \{(\vec u,\vec v)\in\R^{n_p}\times\R^{n_p}\,|\,\exists i\in\{1,\ldots,n_e\}\colon\,(E_{10}\vec u)_i=(E_{10}\vec v)_i=0\}.
\]
A particular Newton derivative can then be chosen as an element of Clarke's generalized Jacobian, 
see \cite{Clarke:1990a}, associated with $\tilde F'$ at $(\vec u, \vec v)$ that is zero at indices corresponding 
to biactive components of $(E_{10}\vec u,E_{10}\vec v)$.
This choice will be used in the proposed method. 

Next, due to the well-known local convergence behavior of Newton's method, the initialization of $\vec{u}$ and $\vec{v}$ for the 
numerical solution of \eqref{eq:NOC_discrete} has to be taken into consideration. 
For that purpose, consider the (infinite-dimensional) problem
\begin{equation}\label{eq:OCPC}\tag{OCNC}
    \left\{
        \begin{aligned}
            \tfrac12\norm{\opE[y]-y_\text d}{L^2(\Omega)}^2+J(u,v)&\,\rightarrow\,\min_{y,u,v}&&&\\
            -\nabla\cdot(\mathbf{C}\nabla y)+\mathbf{a}y&\,=\,\mathbf bu+\mathbf cv && \text{a.e. on } \Omega&\\
            \vec{\mathbf{n}}\cdot(\mathbf{C}\nabla y)&\,=\,0&&\text{a.e. on }\bd \Omega&\\ 
            u,v&\,\geq\,0&&\text{a.e. on }\Omega&
        \end{aligned}
    \right.
\end{equation}
which results from \eqref{eq:OCCC} by omitting the equilibrium condition \eqref{eq:equilibrium_condition} 
and merely imposing nonnegativity constraints. 
Note that \eqref{eq:OCPC} is convex and can be solved globally by combining a penalty algorithm 
and a semismooth Newton method, see \cite{DengMehlitzPruefert2018b}.
The associated global minimizer is uniquely determined.
If its solution already satisfies the equilibrium condition \eqref{eq:equilibrium_condition}, then a global
minimizer of \eqref{eq:OCCC} has already been detected.
The discretized counterpart of \eqref{eq:OCPC} can be derived similarly as stated above. 
The associated (discrete) optimal solution $(\vec y_0,\vec u_0,\vec v_0)$ will be used as the starting vector of the semismooth Newton-type method.
An abstract description of the proposed numerical method for the computational solution of \eqref{eq:OCCC} is presented in \cref{alg:pathfollowing}.
In step \textbf{S2} of this algorithm, $\norm{\cdot}{M}$ denotes a weighted Euclidean norm which represents the discretized
$H^1$-norm, see \cite{DengMehlitzPruefert2018b} for details.
\begin{algorithm}[h]
    \begin{description}
        \item [{S0}] 
            Let $\{\sigma_{k}\}_{k\in\N}$ be a sequence of positive penalty parameters with 
            $\sigma_{k}\to\infty$ as $k\to\infty$. 
            Let a tolerance $\text{eps}>0$ be given. 
            Let $(\vec y_0,\vec u_0,\vec v_0)$ be the (discrete) optimal solution associated with \eqref{eq:OCPC}.
            Compute $\vec p_{0}$ as a solution of the discretized adjoint equation with source $E_{10} \vec y_{0}-\vec y_\text d$. 
            Set $k:=1$.
        \item [{S1}] 
            Solve the discretized KKT system \eqref{eq:NOC_discrete} 
            for fixed $\sigma_{k}$ by a damped, semismooth Newton-type method with starting point 
            $(\vec y_{k-1},\vec{u}_{k-1},\vec{v}_{k-1},\vec p_{k-1})$. 
            Let $(\vec y_k,\vec u_{k},\vec v_{k},\vec p_k)$ be the associated solution.
        \item [{S2}] 
            If $\norm{(\vec u_k,\vec v_k)-(\vec u_{k-1},\vec v_{k-1})}{M}<\text{eps}$ holds true,
            then return $(\vec u_{k}, \vec v_{k})$. 
            Otherwise, set $k:=k+1$ and go to \textbf{S1}. 
    \end{description}

    \caption{Abstract algorithm\label{alg:pathfollowing}}
\end{algorithm}

\subsection{Checking strong stationarity}\label{sec:stationarity_test}

It has to be noted that in step \textbf{S1} of \cref{alg:pathfollowing}, 
one generally only computes critical points to \eqref{eq:discretized_penalized_problem}.
Since the penalty functional $F$ defined in \eqref{eq:penalty} is not convex, 
these cannot be guaranteed to be global minimizers of \eqref{eq:discretized_penalized_problem} 
and therefore the convergence result of \cref{prop:strong_convergence_of_surrogate_solutions} does not apply. 
It is therefore sensible to verify whether the output is at least a strongly stationary point of 
\eqref{eq:OCCC_reduced} in the sense of \cref{cor:consequences_of_S_Stationarity}, 
since the local minimizers of \eqref{eq:OCCC_reduced} can be found among its strongly stationary points. 
Note that available first-order methods for the numerical solution of complementarity problems mainly
compute so-called Clarke- or Mordukhovich-stationary points and that these stationarity notions are weaker
than strong stationarity, 
see, e.g., \cite{HoheiselKanzowSchwartz2013} for a discussion of the finite-dimensional situation.
Thus, checking strong stationarity is recommendable even if a directly discretized version of
\eqref{eq:OCCC_reduced} is solved using the available techniques from finite-dimensional MPCC-theory.
A possible approach for verifying strong stationarity is described in the following.

Let $(y, u,v)\in H^1(\Omega)^3$ be feasible to \eqref{eq:OCCC}.
If this point is a local minimizer, then \cref{cor:consequences_of_S_Stationarity} implies that 
\begin{equation} \label{eq:stationary_equality}
    \dual{y-y_\text{d}}{y}{L^2(\Omega)}
    +\alpha_1\dual{u}{u}{L^2(\Omega)}
    +\alpha_2\dual{v}{v}{L^2(\Omega)}
    +\varepsilon\dual{u}{u}{H^1(\Omega)}
    +\varepsilon\dual{v}{v}{H^1(\Omega)}
    =0
\end{equation}
and that
\begin{equation} \label{eq:stationary_inequality}
    \dual{y-y_\text{d}}{z_y}{L^2(\Omega)}
    +\alpha_1\dual{u}{z_{u}}{L^2(\Omega)}
    +\alpha_2\dual{v}{z_{v}}{L^2(\Omega)}
    +\varepsilon\dual{u}{z_{u}}{H^1(\Omega)}
    +\varepsilon\dual{v}{z_{v}}{H^1(\Omega)}
    \geq 0
\end{equation}
for any pair $(z_u,z_v)\in H^1_+(\Omega)^2$ with 
\[
    \supp z_u\subset I^{+0}(u,v)\cup I^{00}(u,v),\qquad \supp z_v\subset I^{0+}(u,v)\cup I^{00}(u,v),
\]
where $z_y\in H^1(\Omega)$ is the solution of the state equation $\opA[z_y]-\opB[z_u]-\opC[z_v]=0$. 

Using the same discretization technique as described in \cref{sec:discrete_NOC}, 
a discrete counterpart to \eqref{eq:stationary_equality} is
\begin{equation}\label{eq:test_discrete_first_condition}
    \begin{aligned}[t]
        \Theta:=\vec{y}^{\top}E_{10}^\top M_0(1)E_{10}\vec{y}-\vec{y}^{\top}E_{10}^\top M_{0}(1)\vec{y}_\text d
        &+\alpha_1\vec{u}^{\top}M_1(1)\vec{u}
        +\alpha_2\vec{v}^{\top}M_1(1)\vec{v}\\
        &+\varepsilon\vec{u}^{\top}(K(1)+M_1(1))\vec{u}
        +\varepsilon\vec{v}^{\top}(K(1)+M_1(1))\vec{v}
        =0.
    \end{aligned}
\end{equation}
Clearly, a certain tolerance for the violation of \eqref{eq:test_discrete_first_condition}
needs to be imposed in practice.

The numerical verification of condition \eqref{eq:stationary_inequality} 
requires an appropriate choice of discrete test functions $\vec z_u,\vec z_v$ 
for given discretized controls  $(\vec{u},\vec{v})$ in the finite element space $\mathcal{P}^1(\Omega_\Delta)$.
Considering the employed finite element discretization of \eqref{eq:OCCC_reduced}, 
one particular choice is from the set of basis functions associated with $\mathcal P^1(\Omega_\Delta)$. 
Since the support of each of these ``hat functions'' covers all elements adjoining a single vertex, 
a corresponding elementwise approximation of the set  $I^{+0}(u,v)$, $I^{0+}(u,v)$, and $I^{00}(u,v)$,
see \eqref{eq:I+0}, \eqref{eq:I0+}, and \eqref{eq:I00}, respectively, is required as well. 
This can be defined using the projection of $\vec u,\vec v$ from $\mathcal{P}^1(\Omega_\Delta)$ to $\mathcal{P}^0(\Omega_\Delta)$
using the matrix $E_{10}$, which will be denoted by $\vec{u}^0:=E_{10}\vec u$ and $\vec{v}^0:=E_{10}\vec v$, respectively. 
This leads to the corresponding discrete sets
\[
    \begin{aligned}
        I^{+0}(\vec{u},\vec{v})&:=\left\{ i\in\{1,\ldots,n_e\}\,\middle|\, \vec{u}^0_i>0 \text{ and } \vec{v}^0_i=0 \right\},\\
        I^{00}(\vec{u},\vec{v})&:=\left\{ i\in\{1,\ldots,n_e\}\,\middle|\, \vec{u}^0_i=0 \text{ and } \vec{v}^0_i=0 \right\},\\
        I^{0+}(\vec{u},\vec{v}) &:=\left\{ i\in\{1,\ldots,n_e\}\,\middle|\, \vec{u}^0_i=0 \text{ and } \vec{v}^0_i>0 \right\}.
    \end{aligned}
\]
For any pair of basis vectors $(\vec{z}_{u},\vec{z}_{v})$ whose support is contained in $I^{+0}(\vec{u},\vec{v})\cup I^{00}(\vec{u},\vec{v})$
and $I^{0+}(\vec{u},\vec{v})\cup I^{00}(\vec{u},\vec{v})$, respectively, one can then check whether
\begin{equation}\label{eq:test_discrete}
	\begin{aligned}[t]
		\Sigma(\vec{z}_{u},\vec{z}_{v}):=\vec{z}_y^{\top}E_{10}^{\top}M_{0}(1)E_{10}\vec{y} &  
        -\vec{z}_{y}^{\top}E_{10}^{\top} M_{0}(1)\vec{y}_{\text{d}} 
        +\alpha_{1}\vec{u}^{\top}M_{1}(1)\vec{z}_{u}
        +\alpha_{2}\vec{v}^{\top}M_{1}(1)\vec{z}_{v}\\
        & 	+\varepsilon\vec{u}^{\top}\left(K(1)+M_1(1)\right)\vec{z}_{u} 
        +\varepsilon\vec{v}^{\top}\left(K(1)+M_1(1)\right)\vec{z}_{v}\geq0,
    \end{aligned}
\end{equation}
where the state $\vec{z}_{y}$ associated with $(\vec{z}_{u},\vec{z}_{v})$
is obtained via 
\[
    (M_{1}(\vec{a})+K(C)))\vec{z}_{y}=M_{1}(\vec{b})\vec{z}_{u}+M_{1}(\vec{c})\vec{z}_{v}.
\]

In numerical practice, a certain tolerance with respect to negative values of 
$\Sigma(\vec z_u,\vec z_v)$
is necessary since \cref{alg:pathfollowing} involves a penalty procedure and hence yields, 
in general, only \emph{almost} feasible points for \eqref{eq:OCCC}. 
Rather than testing for nonnegativity, it is thus checked whether 
$\Sigma(\vec z_u,\vec z_v)$ is larger than a given negative tolerance.

\section{Numerical examples}\label{sec:numerical_examples}

The proposed numerical method from \cref{sec:numerics} is illustrated by means of three experiments. 
These examples are of academical nature and constructed in such a way that the different features
of the stationarity test are visualized. In the first example, \cref{alg:pathfollowing} computes
globally optimal controls, and thus the results of the corresponding stationarity test provide
a first benchmark for a \emph{numerically passed} stationarity test. 
Examples 2 and 3 provide nontrivial situations where the stationarity test is passed and failed,
respectively. Recall that whenever the stationarity test fails, the considered point cannot be a
local minimizer of the underlying complementarity-constrained program, see \cref{cor:consequences_of_S_Stationarity}.

Let $\Omega:=(0,1)^{2}\subset\R^{2}$. 
For all examples in this section, let $\mathbf C$ be the identity matrix in $\R^{2\times 2}$ and let $\mathbf a\equiv 1$, 
$\mathbf b=\chi_{\Omega_u} $, as well as $\mathbf c=\chi_{\Omega_v}$ hold where
$\Omega_{u}:=\{(x_{1},x_{2})\in\Omega\,|\,x_{2}<0.25\}$ and
$\Omega_{v}:=\{(x_{1},x_{2})\in\Omega\,|\,x_{2}>0.75\}$ are fixed subdomains of $\Omega$.
The values $\alpha_1=\alpha_2=0$ are fixed for this section.
Furthermore, $\varepsilon:=10^{-8}$ is used for all experiments.
The implementation is carried out using the object oriented finite element \textsc{matlab} class library OOPDE, see \cite{Pruefert2015}.

In order to construct examples where the controls are independent of $x_2$, cf. 
\cite[Section~6]{ClasonItoKunisch2016} where parabolic problems were considered
and the controls only depend on time, the problem \eqref{eq:OCCC} will
be equipped with the additional restrictions
\begin{equation}\label{eq:gradient_constraints}
    \partial_{x_2}u=\partial_{x_2}v=0\qquad\text{a.e. on }\Omega.
\end{equation}
These constraints realize controls depending only on $x_{1}$ and being
constant with respect to $x_{2}$ while allowing to use the same finite element space for the discretization of $u$, $v$, and $y$.
Note that the additional constraints do not influence the complementarity constraints (which are now imposed on $\Omega$ rather than $(0,1)$).  
Due to these additional gradient constraints, 
structured grids on the discretized domain $\Omega_\Delta$ are preferentially used for the following examples. 
On unstructured grids, which can be created by local refinement of an arbitrary set of triangles of a structured mesh, 
the use of basis functions from $\mathcal{P}^1(\Omega_\Delta)$ forces the resulting controls to be globally affine,
see \cite[Section~7.2]{DengMehlitzPruefert2018a} for details. This issue can be solved by choosing basis functions from
$\mathcal P^2(\Omega_\Delta)$.
A detailed discussion of optimal control problems with gradient constraints
can be found in \cite{DengMehlitzPruefert2018a}.

To compare results, the solutions of the control problem \eqref{eq:OCPC} without complementarity constraints
(equipped with the additional constraints \eqref{eq:gradient_constraints}) will be 
considered.
Recall that optimal controls $(u,v)\in H^1(\Omega)^2$ of  \eqref{eq:OCPC} additionally fulfilling 
the equilibrium condition \eqref{eq:equilibrium_condition}  solve \eqref{eq:OCCC} as well, 
and that these controls are used as starting points for solving \eqref{eq:OCCC}. 
Since the computed controls are nearly constant with respect to $x_2$, only $u(x_1,0)$ and $v(x_1,0)$ are plotted for the sake of easier comparison.
To evaluate the satisfaction of the complementarity conditions, 
the maximal absolute value of the Fischer--Burmeister function applied componentwise to $(\vec u^0,\vec v^0)$ is reported. 
Furthermore, $\Sigma(\vec z_u,\vec z_v)$ from \eqref{eq:test_discrete} is checked with a tolerance 
\begin{equation}\label{eq:definition_tolerance}
    \mathrm{tol} := 0.01 \left|\min\nolimits_{(\vec z_u,\vec z_v)\text{ feasible test pair}} \Sigma(\vec z_u,\vec z_v)\right|,
\end{equation}
and the number as well as distribution of pairs $(\vec z_u,\vec z_u)$ for which $\Sigma(\vec z_u,\vec z_v)>\mathrm{tol}$ (``numerically positive''), 
$|\Sigma(\vec z_u,\vec z_v)|\leq \mathrm{tol}$ (``numerically zero''), 
or $\Sigma(\vec z_u,\vec z_v)<-\mathrm{tol}$ (``numerically negative'') holds is given.

\paragraph{Example 1}

In this example, the desired state is given by the discontinuous function 
\[
    y_{\text d}(x):=\begin{cases}
        3 & \text{for }x\in[(0.25,0.75)\times(0,0.25)]\cup[(0,0.5)\times(0.75,1)]\\
        1 & \text{otherwise}.
    \end{cases}
\] 
The optimal controls of problem \eqref{eq:OCPC}
are already (numerically) complementary, see \cref{fig:ex4:ocpc}, and thus provide a 
globally optimal solution of \eqref{eq:OCCC}. Correspondingly, 
they coincide with the controls computed for \eqref{eq:OCCC}, see \cref{fig:ex4:occc}, 
for which the maximal absolute value of the Fischer--Burmeister function is $2.08\cdot 10^{-5}$.
With the tolerance chosen as $\mathrm{tol}=2.27\cdot10^{-10}$, $5789$ pairs are labeled as numerically positive, 
$228$ as numerically zero, and $544$ as numerically negative, see \cref{fig:ex4:test}. 
Thus, only $8.3\%$ of all tested pairs belong to the latter category.
Note that $\Theta=-1.65\cdot 10^{-7}$ holds for the constant defined in \eqref{eq:test_discrete_first_condition}.

Observing that \cref{alg:pathfollowing} computes the globally optimal solution
of \eqref{eq:OCCC} in this example, the above data represent an
\emph{approximately} passed stationarity test. 
\definecolor{mycolor1}{rgb}{0.231674, 0.318106, 0.544834}
\definecolor{mycolor2}{rgb}{0.369214, 0.788888, 0.382914}
\renewcommand{\topfraction}{0.9}
\begin{figure}[t]
    \centering
    \begin{subfigure}[t]{0.49\textwidth}
        \centering
        \begin{tikzpicture}

\begin{axis}[%
width=\textwidth,
xmin=0,
xmax=1,
xlabel=$x_1$,
ymin=-1,
ymax=35,
legend style={legend pos=north east,legend style={draw=none}},
]
\addplot [color=mycolor1,solid,line width=1.5pt]
  table[row sep=crcr]{%
0	-0.000774546750747618\\
0.199999999999999	-0.000519851680927275\\
0.425000000000001	0.000100868162412837\\
0.4375	3.12651410484792\\
0.449999999999999	8.48056739942542\\
0.462499999999999	14.45492158816\\
0.475000000000001	19.7944506583162\\
0.487500000000001	23.5885232388268\\
0.5	25.2649490419579\\
0.512499999999999	24.5860091277224\\
0.524999999999999	21.6468785300207\\
0.537500000000001	16.876575621154\\
0.5625	5.25053815978681\\
0.574999999999999	0.963715871686624\\
0.587499999999999	2.8793481913425e-05\\
0.712499999999999	-0.000217157709080595\\
0.962499999999999	-0.000739326266646856\\
1	-0.000749143266503438\\
};
\addlegendentry{$u(x_1,0)$}

\addplot [color=mycolor2,solid,line width=1.5pt]
  table[row sep=crcr]{%
0	33.1689703011959\\
0.0125000000000028	32.6694577326451\\
0.0249999999999986	31.9890517080171\\
0.0375000000000014	31.102592052781\\
0.0499999999999972	29.9724401460156\\
0.0625	28.5530559648829\\
0.0750000000000028	26.797397042952\\
0.0874999999999986	24.665042837401\\
0.100000000000001	22.1319449166012\\
0.112499999999997	19.201688807389\\
0.125	15.9181285835865\\
0.137500000000003	12.3792279908477\\
0.149999999999999	8.75191361105632\\
0.162500000000001	5.28771579191143\\
0.174999999999997	2.33895709071454\\
0.1875	0.37524039861318\\
0.200000000000003	-4.53290486746027e-06\\
0.649999999999999	-1.1568866618461e-05\\
1	-1.27436038326323e-05\\
};
\addlegendentry{$v(x_1,0)$}

\end{axis}
\end{tikzpicture}%
        \caption{solution of \eqref{eq:OCPC}}\label{fig:ex4:ocpc}
    \end{subfigure}
    \hfill
    \begin{subfigure}[t]{0.49\textwidth}
        \centering
        \begin{tikzpicture}

\begin{axis}[%
width=\textwidth,
xmin=0,
xmax=1,
xlabel=$x_1$,
ymin=-1,
ymax=35,
legend style={legend pos=north east,legend style={draw=none}},
]
\addplot [color=mycolor1,solid,line width=1.5pt]
  table[row sep=crcr]{%
0	-1.4659823762031e-05\\
0.425000000000001	4.35463544405934e-05\\
0.4375	3.12663519969359\\
0.449999999999999	8.48039739684748\\
0.462499999999999	14.4540819373248\\
0.475000000000001	19.7926588911085\\
0.487500000000001	23.5856066201274\\
0.5	25.2608591850069\\
0.512499999999999	24.5808359129384\\
0.524999999999999	21.6408669690091\\
0.537500000000001	16.8701418009202\\
0.5625	5.24527404423822\\
0.574999999999999	0.960466453116439\\
0.587499999999999	3.67398427556509e-05\\
1	-9.67675910246157e-06\\
};
\addlegendentry{$u(x_1,0)$}

\addplot [color=mycolor2,solid,line width=1.5pt]
  table[row sep=crcr]{%
0	33.1744754965281\\
0.0125000000000028	32.6747958294901\\
0.0249999999999986	31.9939088400311\\
0.0375000000000014	31.1066666780404\\
0.0499999999999972	29.975457467639\\
0.0625	28.5547793533139\\
0.0750000000000028	26.7976409964451\\
0.0874999999999986	24.6636892136495\\
0.100000000000001	22.1289599098897\\
0.112499999999997	19.197142419403\\
0.125	15.9122160267819\\
0.137500000000003	12.3722922742884\\
0.149999999999999	8.74447023059745\\
0.162500000000001	5.28047776691635\\
0.174999999999997	2.33286065672185\\
0.1875	0.371470611863373\\
0.200000000000003	2.59213634876687e-05\\
0.5625	5.97998450757586e-09\\
1	3.60026366124089e-05\\
};
\addlegendentry{$v(x_1,0)$}

\end{axis}
\end{tikzpicture}%
        \caption{solution of \eqref{eq:OCCC}}\label{fig:ex4:occc}
    \end{subfigure}
    \caption{Example 1: computed controls}
    \label{fig:ex4_sol} 
\end{figure}
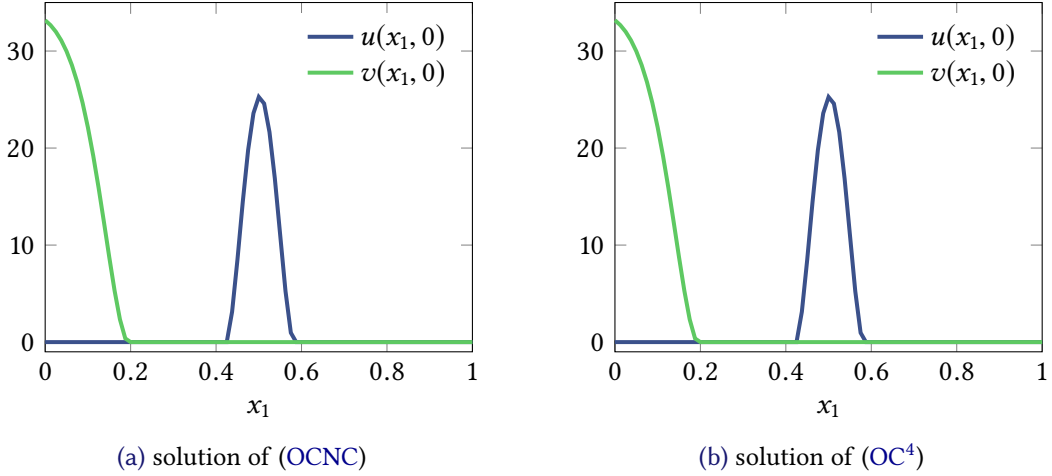
\begin{figure}[t]
    \centering
    \begin{subfigure}[t]{0.49\textwidth}
        \centering
        \begin{tikzpicture}

\begin{axis}[
height=0.75\textwidth,
colorbar,
colorbar style={
ytick={-0.000292517664368325,-0.000306427102567224,-0.000314563602319679,-0.000246311510831302,-0.000260220949030201,-0.000268357448782656,-0.0002741303872291,-0.000278608226169426,-0.000282266886981555,-0.000285360240622038,-0.000288039825428,-0.00029040338673401,-0.000200105357294278,-0.000214014795493178,-0.000222151295245632,-0.000227924233692077,-0.000232402072632402,-0.000236060733444532,-0.000239154087085015,-0.000241833671890976,-0.000244197233196986,-0.000153899203757255,-0.000167808641956154,-0.000175945141708609,-0.000181718080155054,-0.000186195919095379,-0.000189854579907508,-0.000192947933547991,-0.000195627518353953,-0.000197991079659963,-0.000107693050220232,-0.000121602488419131,-0.000129738988171586,-0.00013551192661803,-0.000139989765558356,-0.000143648426370485,-0.000146741780010968,-0.00014942136481693,-0.00015178492612294,-6.14868966832083e-05,-7.53963348821076e-05,-8.35328346345622e-05,-8.9305773081007e-05,-9.37836120213322e-05,-9.74422728334616e-05,-0.000100535626473945,-0.000103215211279906,-0.000105578772585916,-1.39748909141323e-05,-2.79497818282647e-05,-3.73266810975388e-05,-4.30996195439837e-05,-4.75774584843089e-05,-5.12361192964383e-05,-5.43294729369212e-05,-5.70090577428831e-05,-5.93726190488927e-05,0,1.39748909141323e-05,2.79497818282646e-05,3.73266810975388e-05,4.30996195439837e-05,4.75774584843089e-05,5.12361192964383e-05,5.43294729369211e-05,5.70090577428831e-05,5.93726190488928e-05,6.14868966832082e-05,7.53963348821076e-05,8.35328346345622e-05,8.9305773081007e-05,9.37836120213322e-05,9.74422728334616e-05,0.000100535626473944,0.000103215211279906,0.000105578772585916,0.000107693050220232,0.000121602488419131,0.000129738988171586,0.00013551192661803,0.000139989765558356,0.000143648426370485,0.000146741780010968,0.00014942136481693,0.00015178492612294,0.000153899203757255,0.000167808641956154,0.000175945141708609,0.000181718080155054,0.000186195919095379,0.000189854579907508,0.000192947933547991,0.000195627518353953,0.000197991079659963,0.000200105357294278,0.000214014795493178,0.000222151295245632,0.000227924233692077,0.000232402072632402,0.000236060733444532,0.000239154087085015,0.000241833671890976,0.000244197233196986,0.000246311510831302,0.000260220949030201,0.000268357448782656,0.0002741303872291,0.000278608226169426,0.000282266886981555,0.000285360240622038,0.000288039825428,0.00029040338673401,0.000292517664368325,0.000306427102567224,0.000314563602319679},
yticklabels={${-10^{-4}}$,,,${-10^{-5}}$,,,,,,,,,${-10^{-6}}$,,,,,,,,,${-10^{-7}}$,,,,,,,,,${-10^{-8}}$,,,,,,,,,${-10^{-9}}$,,,,,,,,,${-10^{-10}}$,,,,,,,,,,${10^{-10}}$,,,,,,,,,${10^{-9}}$,,,,,,,,,${10^{-8}}$,,,,,,,,,${10^{-7}}$,,,,,,,,,${10^{-6}}$,,,,,,,,,${10^{-5}}$,,,,,,,,,${10^{-4}}$,,},
scaled y ticks=false,
tick label style={font=\scriptsize},
},
colormap={mymap}{[1pt]
  rgb(0pt)=(0.176470588235294,0,0.294117647058824);
  rgb(1pt)=(0.329411764705882,0.152941176470588,0.533333333333333);
  rgb(2pt)=(0.501960784313725,0.450980392156863,0.674509803921569);
  rgb(3pt)=(0.698039215686274,0.670588235294118,0.823529411764706);
  rgb(4pt)=(0.847058823529412,0.854901960784314,0.92156862745098);
  rgb(5pt)=(0.968627450980392,0.968627450980392,0.968627450980392);
  rgb(6pt)=(0.996078431372549,0.87843137254902,0.713725490196078);
  rgb(7pt)=(0.992156862745098,0.72156862745098,0.388235294117647);
  rgb(8pt)=(0.87843137254902,0.509803921568627,0.0784313725490196);
  rgb(9pt)=(0.701960784313725,0.345098039215686,0.0235294117647059);
  rgb(10pt)=(0.498039215686275,0.231372549019608,0.0313725490196078)
},
point meta max=0.00031557956563535,
point meta min=-0.00031557956563535,
xmin=0,
xmax=81,
xlabel=$\vec z_u$,
ylabel=$\vec z_v$,
ymin=0,
ymax=81,
]
\addplot graphics [includegraphics cmd=\pgfimage,xmin=0, xmax=81, ymin=0, ymax=81] {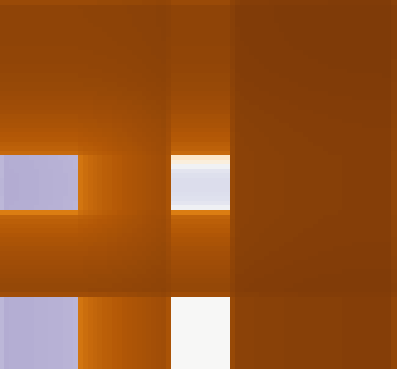};
\end{axis}

\end{tikzpicture}
        \caption{$\Sigma(\vec z_u,\vec z_v)$}\label{fig:ex4:comp}
    \end{subfigure}
    \hfill
    \begin{subfigure}[t]{0.49\textwidth}
        \centering
        \begin{tikzpicture}

\begin{axis}[
height=0.75\textwidth,
point meta max=0.00031557956563535,
point meta min=-2.27162053661571e-08,
xlabel=$\vec z_u$,
ylabel=$\vec z_v$,
xmin=0, xmax=81,
ymin=0, ymax=81,
axis on top,
major tick length=0pt
]
\addplot graphics [includegraphics cmd=\pgfimage,xmin=0, xmax=81, ymin=0, ymax=81] {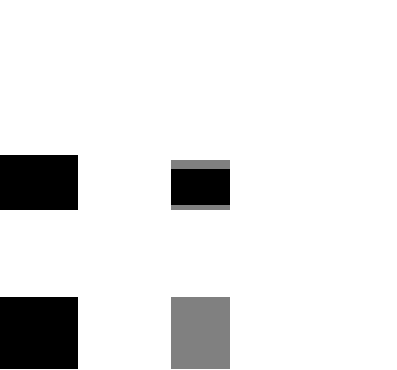};
\end{axis}

\end{tikzpicture}
        \caption{pairs marked numerically positive (white), numerically zero (gray), numerically negative (black)}\label{fig:ex4:test}
    \end{subfigure}
    \caption{Example 1: values of stationarity test and distribution of failed pairs}
    \label{fig:ex4_stat}
\end{figure}

\paragraph{Example 2}
Here, the desired state is chosen to be the (weak) solution of the elliptic boundary value problem
\[
    \left\{
        \begin{aligned}
            -\Delta y(x) &\,=\,0&\qquad&\text{a.e. on }\Omega&\\
            y(x) &\,=\,2\max\{0;x_1\cos(0.75\pi x_{1})\}&&\text{a.e. on }\Gamma_{1}&\\
            y(x) &\,=\,0.25&&\text{a.e. on }\Gamma_{2}&\\
            \vec{\mathbf n}(x)\cdot\nabla y(x)&\,=\,0&&\text{a.e. on }\Gamma_3&
        \end{aligned}
    \right.
\]
where $\Gamma_{1}:=[0,1]\times\{0\}$, 
$\Gamma_{2}:=[0,1]\times\{1\}$, and $\Gamma_3:=\{0,1\}\times[0,1]$ are fixed.
The optimal controls of the associated problem \eqref{eq:OCPC} do not fulfill the complementarity
condition but already provide a biactive set, see \cref{fig:ex1:ocpc}. 
On the other hand, the computed solution for \eqref{eq:OCCC} approximately satisfies 
the complementarity condition, see \cref{fig:ex1:occc}, 
with a maximal absolute value of the Fischer--Burmeister function of approximately $3.58\cdot 10^{-6}$.
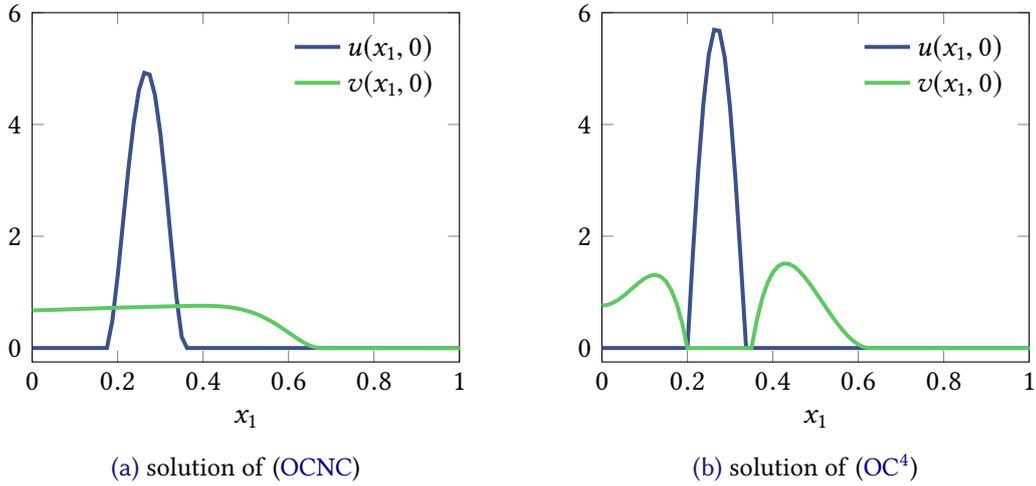
\begin{figure}[t]
    \centering
    \begin{subfigure}[t]{0.49\textwidth}
        \centering
        \begin{tikzpicture}

\begin{axis}[%
width=\textwidth,
xmin=0,
xmax=1,
xlabel=$x_1$,
ymin=-0.25,
ymax=6,
legend style={legend pos=north east,legend style={draw=none}},
]
\addplot [color=mycolor1,solid,line width=1.5pt]
  table[row sep=crcr]{%
0	-1.63765944636296e-05\\
0.175	1.0999089676389e-05\\
0.1875	0.4816651400351\\
0.2	1.28339287233451\\
0.225	3.18856480125628\\
0.2375	4.01619005124615\\
0.25	4.61823539600429\\
0.2625	4.92354204369873\\
0.275	4.89282565416787\\
0.2875	4.52156087792864\\
0.3	3.84289677596361\\
0.3125	2.93058600615304\\
0.325	1.90188404830439\\
0.3375	0.920350989593821\\
0.35	0.198469998081544\\
0.3625	-5.16472862432948e-06\\
0.9125	-0.000193092357813285\\
1	-0.000200446846562841\\
};
\addlegendentry{$u(x_1,0)$}

\addplot [color=mycolor2,solid,line width=1.5pt]
  table[row sep=crcr]{%
0	0.672725235582645\\
0.0249999999999999	0.67760239415077\\
0.05	0.683090066939403\\
0.075	0.6891328941114\\
0.1125	0.69891924058434\\
0.1625	0.712150369866321\\
0.1875	0.71833112465793\\
0.2125	0.724027596734417\\
0.2375	0.729235271762426\\
0.275	0.73641398929213\\
0.35	0.750244346971682\\
0.3625	0.752229730126364\\
0.375	0.753840635057313\\
0.3875	0.754860810502356\\
0.4	0.755011995880262\\
0.4125	0.753950209497013\\
0.425	0.751264780393593\\
0.4375	0.746480842836137\\
0.45	0.73906611844849\\
0.4625	0.72844279617894\\
0.475	0.714005245043241\\
0.4875	0.695144240511689\\
0.5	0.671278223300286\\
0.5125	0.641891923703708\\
0.525	0.606582526808709\\
0.5375	0.565113123620892\\
0.55	0.517472977171075\\
0.5625	0.463943673769318\\
0.575	0.405169847918021\\
0.5875	0.342232708117858\\
0.6125	0.21081928892215\\
0.625	0.14734327370702\\
0.6375	0.0898314077086013\\
0.65	0.0425773422016047\\
0.6625	0.0106676811871078\\
0.675	1.3390608399888e-08\\
1	-1.17363287754912e-07\\
};
\addlegendentry{$v(x_1,0)$}

\end{axis}
\end{tikzpicture}%
        \caption{solution of \eqref{eq:OCPC}}\label{fig:ex1:ocpc}
    \end{subfigure}
    \hfill
    \begin{subfigure}[t]{0.49\textwidth}
        \centering
        \begin{tikzpicture}

\begin{axis}[%
width=\textwidth,
xmin=0,
xmax=1,
xlabel=$x_1$,
ymin=-0.25,
ymax=6,
legend style={legend pos=north east,legend style={draw=none}},
]
\addplot [color=mycolor1,solid,line width=1.5pt]
  table[row sep=crcr]{%
0	-3.89412155854529e-07\\
0.2	3.58995323246347e-07\\
0.2125	1.64417508706178\\
0.225	3.14362046660565\\
0.2375	4.37701458695869\\
0.25	5.2491829414291\\
0.2625	5.69378833222503\\
0.275	5.6760386796025\\
0.2875	5.19546878402539\\
0.3	4.28881770606541\\
0.3125	3.03298489460585\\
0.325	1.54801176934695\\
0.3375	4.91545925562775e-06\\
1	-3.20789147867373e-06\\
};
\addlegendentry{$u(x_1,0)$}

\addplot [color=mycolor2,solid,line width=1.5pt]
  table[row sep=crcr]{%
0	0.757943026761256\\
0.0125	0.771364241897604\\
0.0249999999999999	0.80587082415307\\
0.0375000000000001	0.859401410534695\\
0.05	0.928518310988721\\
0.0625	1.00841161583916\\
0.075	1.09291146791472\\
0.0874999999999999	1.17451442173782\\
0.1	1.24443058270146\\
0.1125	1.29265894146864\\
0.125	1.30809803281611\\
0.1375	1.27869824402024\\
0.15	1.19166057230483\\
0.1625	1.03368352643541\\
0.175	0.791256812706355\\
0.1875	0.450995141476369\\
0.2	7.01996323293486e-08\\
0.35	1.76010606312005e-07\\
0.3625	0.497856339201983\\
0.375	0.881061846252264\\
0.3875	1.16169871040737\\
0.4	1.35221997805548\\
0.4125	1.46508272687294\\
0.425	1.51241613834528\\
0.4375	1.50574274977001\\
0.45	1.455762110671\\
0.4625	1.37220114892352\\
0.475	1.26373075851855\\
0.4875	1.13794522089381\\
0.5	1.00139823968345\\
0.5375	0.579172164299843\\
0.55	0.448054398487506\\
0.5625	0.32751312399875\\
0.575	0.220714106899987\\
0.5875	0.130888957748941\\
0.6	0.061506984249762\\
0.6125	0.0164278024502642\\
0.625	2.99137108024095e-05\\
1	1.7423532047145e-05\\
};
\addlegendentry{$v(x_1,0)$}

\end{axis}
\end{tikzpicture}%
        \caption{solution of \eqref{eq:OCCC}}\label{fig:ex1:occc}
    \end{subfigure}
    \caption{Example 2: computed controls}
    \label{fig:ex1_sol} 
\end{figure}
\begin{figure}
    \centering
    \begin{subfigure}[t]{0.49\textwidth}
        \centering
        \begin{tikzpicture}

\begin{axis}[
height=0.75\textwidth,
colorbar,
colorbar style={%
ytick={-2.59634854647523e-05,-2.82104925315197e-05,-1.84990895602472e-05,-2.07460966270146e-05,-2.2060511499929e-05,-2.2993103693782e-05,-2.37164783979849e-05,-2.43075185666964e-05,-2.48072359090349e-05,-2.52401107605494e-05,-2.56219334396108e-05,-1.10346936557421e-05,-1.32817007225095e-05,-1.45961155954239e-05,-1.55287077892769e-05,-1.62520824934798e-05,-1.68431226621913e-05,-1.73428400045298e-05,-1.77757148560443e-05,-1.81575375351057e-05,-3.17136076291969e-06,-5.81730481800446e-06,-7.13171969091884e-06,-8.06431188477186e-06,-8.78768658897474e-06,-9.37872675768623e-06,-9.87844410002475e-06,-1.03113189515393e-05,-1.06931416306006e-05,0,3.17136076291969e-06,5.81730481800447e-06,7.13171969091884e-06,8.06431188477187e-06,8.78768658897473e-06,9.37872675768624e-06,9.87844410002475e-06,1.03113189515393e-05,1.06931416306006e-05,1.10346936557421e-05,1.32817007225095e-05,1.45961155954239e-05,1.55287077892769e-05,1.62520824934798e-05,1.68431226621913e-05,1.73428400045298e-05,1.77757148560443e-05,1.81575375351057e-05,1.84990895602472e-05,2.07460966270146e-05,2.2060511499929e-05,2.2993103693782e-05,2.37164783979849e-05,2.43075185666964e-05,2.48072359090349e-05,2.52401107605494e-05,2.56219334396107e-05,2.59634854647523e-05,2.82104925315197e-05},
yticklabels={${-10^{-5}}$,,${-10^{-6}}$,,,,,,,,,${-10^{-7}}$,,,,,,,,,${-10^{-8}}$,,,,,,,,,${0}$,${10^{-8}}$,,,,,,,,,${10^{-7}}$,,,,,,,,,${10^{-6}}$,,,,,,,,,${10^{-5}}$,},
scaled y ticks=false,
tick label style={font=\scriptsize},
},
colormap={mymap}{[1pt]
  rgb(0pt)=(0.176470588235294,0,0.294117647058824);
  rgb(1pt)=(0.329411764705882,0.152941176470588,0.533333333333333);
  rgb(2pt)=(0.501960784313725,0.450980392156863,0.674509803921569);
  rgb(3pt)=(0.698039215686274,0.670588235294118,0.823529411764706);
  rgb(4pt)=(0.847058823529412,0.854901960784314,0.92156862745098);
  rgb(5pt)=(0.968627450980392,0.968627450980392,0.968627450980392);
  rgb(6pt)=(0.996078431372549,0.87843137254902,0.713725490196078);
  rgb(7pt)=(0.992156862745098,0.72156862745098,0.388235294117647);
  rgb(8pt)=(0.87843137254902,0.509803921568627,0.0784313725490196);
  rgb(9pt)=(0.701960784313725,0.345098039215686,0.0235294117647059);
  rgb(10pt)=(0.498039215686275,0.231372549019608,0.0313725490196078)
},
point meta max=2.94667692178952e-05,
point meta min=-2.94667692178952e-05,
xmin=0,
xmax=81,
xlabel=$\vec z_u$,
ylabel=$\vec z_v$,
ymin=0,
ymax=81,
]
\addplot graphics [includegraphics cmd=\pgfimage,xmin=0, xmax=81, ymin=0, ymax=81] {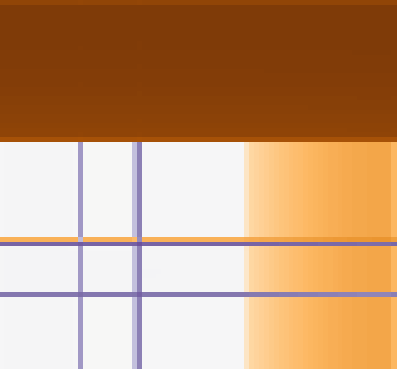};
\end{axis}

\end{tikzpicture}
        \caption{$\Sigma(\vec z_u,\vec z_v)$}\label{fig:ex1:comp}
    \end{subfigure}
    \hfill
    \begin{subfigure}[t]{0.49\textwidth}
        \centering
        \begin{tikzpicture}

\begin{axis}[
height=0.75\textwidth,
point meta max=2.94667692178952e-05,
point meta min=-1.61708710207973e-06,
xlabel=$\vec z_u$,
ylabel=$\vec z_v$,
xmin=0, xmax=81,
ymin=0, ymax=81,
axis lines=box,
axis on top,
major tick length=0pt
]
\addplot graphics [includegraphics cmd=\pgfimage,xmin=0, xmax=81, ymin=0, ymax=81] {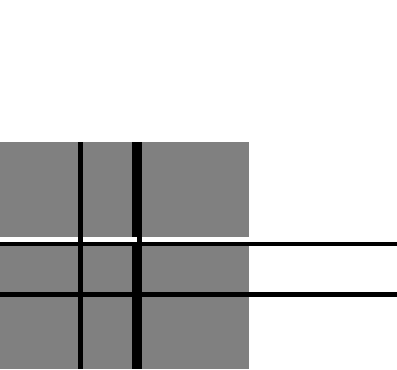};
\end{axis}

\end{tikzpicture}
        \caption{pairs marked numerically positive (white), numerically zero (gray), numerically negative (black)}\label{fig:ex1:test}
    \end{subfigure}
    \caption{Example 2: values of stationarity test and distribution of failed pairs}
    \label{fig:ex1_stat}
\end{figure}
The minimal value of $\Sigma(\vec z_u,\vec z_v)$ was approximately $-1.62\cdot 10^{-6}$, cf. \cref{fig:ex1:comp}. 
Accordingly, the tolerance for the stationarity test was chosen as $\mathrm{tol} = 1.617\cdot 10^{-8}$. 
This leads to $4000$ pairs $(\vec z_u,\vec z_v)$  marked as ``numerically positive'', $2256$ as ``numerically zero'', 
and $305$ as ``numerically negative'' and thus failing the strong stationarity test \eqref{eq:test_discrete}, 
see \cref{fig:ex1:test}. 
These amount to approximately $4.7\%$ of the total number $6561$ of pairs.
Note that pairs where the stationarity test fails correlate with those basis functions 
associated with nodes where the subdomains $I^{+0}(\vec u,\vec v)$ and $I^{0+}(\vec u,\vec v)$ meet. 
Finally, $\Theta=-2.01\cdot 10^{-9}$ holds.

\paragraph{Example 3}

In the last experiment, the desired state is given by $y_{\text d}\equiv 1.5$. 
The optimal controls for the problem
\eqref{eq:OCPC} are nearly constant functions, see \cref{fig:ex2:ocpc}. 
The controls for the problem \eqref{eq:OCCC} computed via \cref{alg:pathfollowing} are complementary, see \cref{fig:ex2:occc}.
The maximal absolute value of the Fischer--Burmeister function is $2.02\cdot 10^{-6}$.
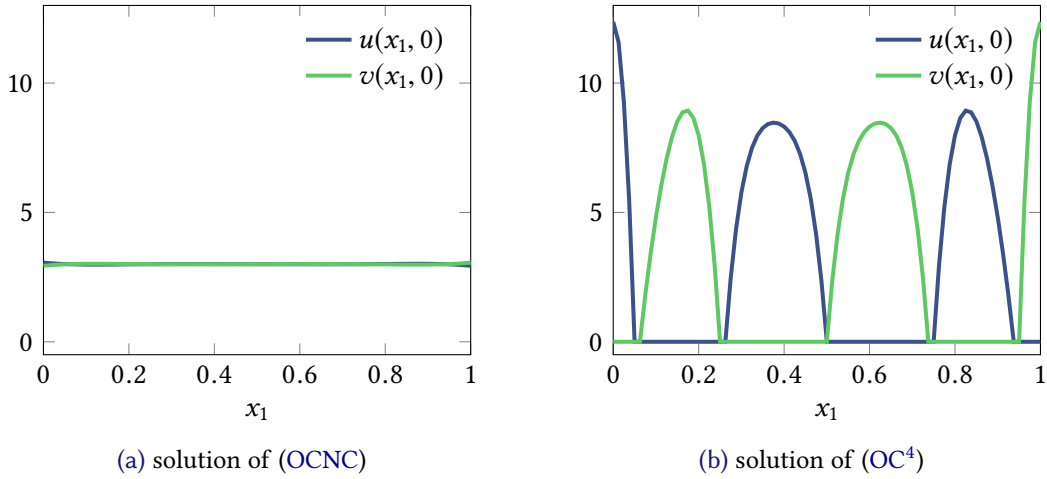
\begin{figure}[t]
    \centering
    \begin{subfigure}[t]{0.49\textwidth}
        \centering
        \begin{tikzpicture}

\begin{axis}[%
width=\textwidth,
xmin=0,
xmax=1,
xlabel=$x_1$,
ymin=-0.5,
ymax=13,
legend style={legend pos=north east,legend style={draw=none}},
]
\addplot [color=mycolor1,solid,line width=1.5pt]
  table[row sep=crcr]{%
0	3.05925279476933\\
0.0125000000000002	3.04171009535151\\
0.0249999999999999	3.02674679437806\\
0.0375000000000001	3.014277484849\\
0.0499999999999998	3.00416919534134\\
0.0625	2.99624898802366\\
0.0750000000000002	2.99031339351996\\
0.0874999999999999	2.98613814648335\\
0.1	2.98348766240189\\
0.1125	2.98212375328258\\
0.125	2.98181327395763\\
0.1375	2.98233444858531\\
0.15	2.98348200633125\\
0.1625	2.98507078997425\\
0.175	2.9869381939036\\
0.2125	2.99294389692892\\
0.225	2.994774371208\\
0.2375	2.99642007370803\\
0.25	2.99784978444251\\
0.2625	2.99904780558509\\
0.275	3.00001133400705\\
0.2875	3.00074787938697\\
0.3	3.00127296972399\\
0.3125	3.00160776149904\\
0.325	3.00177697329918\\
0.3375	3.00180708577962\\
0.35	3.00172483010473\\
0.3625	3.0015559011934\\
0.3875	3.00105020545394\\
0.475	2.99901436145631\\
0.55	2.99747462767869\\
0.5875	2.99657673916283\\
0.625	2.99569165374855\\
0.6375	2.99545811071237\\
0.65	2.99528764903132\\
0.6625	2.995204050717\\
0.675	2.99523314955686\\
0.6875	2.99540176915307\\
0.7	2.99573650730691\\
0.7125	2.99626217089547\\
0.725	2.99700002650847\\
0.7375	2.99796564616644\\
0.75	2.99916663501953\\
0.7625	3.00060024725949\\
0.775	3.00225072723794\\
0.7875	3.00408682136934\\
0.8	3.00605930511755\\
0.825	3.01011332740459\\
0.8375	3.01198812071895\\
0.85	3.01358385894552\\
0.8625	3.01473744109993\\
0.875	3.01526310595507\\
0.8875	3.01495488444132\\
0.9	3.01359024337435\\
0.9125	3.01093519634799\\
0.925	3.00675066658869\\
0.9375	3.00080016823484\\
0.95	2.99285855414871\\
0.9625	2.98272146749666\\
0.975	2.97021517270065\\
0.9875	2.95520592676277\\
1	2.93760774694376\\
};
\addlegendentry{$u(x_1,0)$}

\addplot [color=mycolor2,solid,line width=1.5pt]
  table[row sep=crcr]{%
0	2.93761480358703\\
0.0125000000000002	2.95521215643124\\
0.0249999999999999	2.9702214048975\\
0.0375000000000001	2.98272768564218\\
0.0499999999999998	2.99286477790119\\
0.0625	3.00080642111077\\
0.0750000000000002	3.00675694355298\\
0.0874999999999999	3.01094149805188\\
0.1	3.01359655946869\\
0.1125	3.01496120174114\\
0.125	3.01526943327347\\
0.1375	3.01474375448311\\
0.15	3.01359017150511\\
0.1625	3.01199443189073\\
0.175	3.01011963851145\\
0.2125	3.00409317633869\\
0.225	3.00225706430117\\
0.2375	3.00060656263746\\
0.25	2.99917294290594\\
0.2625	2.99797194811131\\
0.275	2.99700634654625\\
0.2875	2.99626848091825\\
0.3	2.99574277144858\\
0.3125	2.99540801173757\\
0.325	2.99523938857932\\
0.3375	2.99521030850331\\
0.35	2.9952938922772\\
0.3625	2.99546433645371\\
0.3875	2.99597345021322\\
0.4625	2.9977557724151\\
0.575	3.00014318261428\\
0.625	3.00133040322698\\
0.6375	3.0015622580444\\
0.65	3.00173115598441\\
0.6625	3.00181339577718\\
0.675	3.00178328321431\\
0.6875	3.00161407443469\\
0.7	3.00127928695965\\
0.7125	3.00075418689691\\
0.725	3.00001762837335\\
0.7375	2.99905411603692\\
0.75	2.99785608998218\\
0.7625	2.99642637602784\\
0.775	2.99478067169083\\
0.7875	2.9929501970228\\
0.8	2.99098411436689\\
0.825	2.98694444771689\\
0.8375	2.98507702050617\\
0.85	2.98348824489182\\
0.8625	2.98234068693028\\
0.875	2.98181949973065\\
0.8875	2.98213000915595\\
0.9	2.98349394430295\\
0.9125	2.98614445906982\\
0.925	2.99031974596626\\
0.9375	2.99625536293892\\
0.95	3.0041755838302\\
0.9625	3.01428386614541\\
0.975	3.02675315148494\\
0.9875	3.04171644937375\\
1	3.05925832322848\\
};
\addlegendentry{$v(x_1,0)$}

\end{axis}
\end{tikzpicture}%
        \caption{solution of \eqref{eq:OCPC}}\label{fig:ex2:ocpc}
    \end{subfigure}
    \hfill
    \begin{subfigure}[t]{0.49\textwidth}
        \centering
        \begin{tikzpicture}

\begin{axis}[%
width=\textwidth,
xmin=0,
xmax=1,
xlabel=$x_1$,
ymin=-0.5,
ymax=13,
legend style={legend pos=north east,legend style={draw=none}},
]
\addplot [color=mycolor1,solid,line width=1.5pt]
  table[row sep=crcr]{%
0	12.3657539523269\\
0.0124999999999993	11.5813451870105\\
0.0250000000000004	9.26192728774273\\
0.0374999999999996	5.40336381488639\\
0.0500000000000007	1.00029578860017e-05\\
0.262499999999999	4.26302893075103e-06\\
0.275	2.46750105904479\\
0.2875	4.36032516091317\\
0.300000000000001	5.77283772741934\\
0.3125	6.79373330099965\\
0.324999999999999	7.50380838242849\\
0.3375	7.97392690148423\\
0.35	8.26326209757148\\
0.362500000000001	8.41786903988326\\
0.375	8.46961930582469\\
0.387499999999999	8.4355168097466\\
0.4	8.31740604162014\\
0.4125	8.10207375287712\\
0.425000000000001	7.76174342311218\\
0.4375	7.25495316073645\\
0.449999999999999	6.52780366352194\\
0.4625	5.51555452118935\\
0.475	4.14453250220068\\
0.487500000000001	2.33429953914506\\
0.5	2.33023186524406e-07\\
0.75	4.7629752426559e-06\\
0.762499999999999	2.93793171161431\\
0.775	5.19671295754257\\
0.7875	6.85358061685019\\
0.800000000000001	7.98309265125263\\
0.8125	8.65509187727172\\
0.824999999999999	8.93296685266104\\
0.8375	8.87230741128275\\
0.85	8.52000688354622\\
0.862500000000001	7.91383043027835\\
0.875	7.08244187155882\\
0.887499999999999	6.04585908212986\\
0.9	4.81628727452416\\
0.9125	3.39926336290677\\
0.925000000000001	1.79502671586163\\
0.9375	2.54599758005014e-05\\
1	-1.72779770934994e-07\\
};
\addlegendentry{$u(x_1,0)$}

\addplot [color=mycolor2,solid,line width=1.5pt]
  table[row sep=crcr]{%
0	-1.83191513514203e-07\\
0.0625	6.2749999028e-06\\
0.0749999999999993	1.79502771100866\\
0.0875000000000004	3.39926451251249\\
0.0999999999999996	4.81628849840033\\
0.112500000000001	6.04586025221918\\
0.125	7.08244292178752\\
0.137499999999999	7.91383105865923\\
0.15	8.52000703412061\\
0.1625	8.87230702233799\\
0.175000000000001	8.93296585641985\\
0.1875	8.65509020422291\\
0.199999999999999	7.98309033991144\\
0.2125	6.85357763652616\\
0.225	5.19670927648667\\
0.237500000000001	2.93792729054892\\
0.25	4.51593994554855e-06\\
0.5	4.10975108167122e-07\\
0.512499999999999	2.33429678873987\\
0.525	4.1445305711215\\
0.5375	5.51555346370156\\
0.550000000000001	6.52780355160604\\
0.5625	7.25495382269674\\
0.574999999999999	7.76174463359503\\
0.5875	8.10207529601879\\
0.6	8.31740792760717\\
0.612500000000001	8.43551904177067\\
0.625	8.46962159298031\\
0.637499999999999	8.41787110568223\\
0.65	8.26326365627043\\
0.6625	7.97392796834404\\
0.675000000000001	7.50380898940213\\
0.6875	6.79373319394689\\
0.699999999999999	5.77283672991612\\
0.7125	4.36032306455092\\
0.725	2.46749789147512\\
0.737500000000001	4.11861624982635e-06\\
0.949999999999999	1.03068750156865e-06\\
0.9625	5.40335030812917\\
0.975	9.26191358706357\\
0.987500000000001	11.5813311547951\\
1	12.3657395391657\\
};
\addlegendentry{$v(x_1,0)$}

\end{axis}
\end{tikzpicture}%
        \caption{solution of \eqref{eq:OCCC}}\label{fig:ex2:occc}
    \end{subfigure}
    \caption{Example 3: computed controls}
    \label{fig:ex2_sol} 
\end{figure}
\begin{figure}[t]
    \centering
    \begin{subfigure}[t]{0.49\textwidth}
        \centering
        \begin{tikzpicture}

\begin{axis}[
height=0.75\textwidth,
colorbar,
colorbar style={%
ytick={-1.09337170193383e-05,-6.79329339563857e-06,-8.03968510112794e-06,-8.76877751004911e-06,-9.28607680661731e-06,-9.68732531384891e-06,-1.00151692155385e-05,-1.02923572848559e-05,-1.05324685121067e-05,-1.07442616244596e-05,-2.55687594664287e-06,-3.89926147742824e-06,-4.62835388634941e-06,-5.14565318291761e-06,-5.54690169014921e-06,-5.87474559183877e-06,-6.15193366115622e-06,-6.39204488840697e-06,-6.60383800075994e-06,0,2.55687594664286e-06,3.89926147742824e-06,4.62835388634941e-06,5.14565318291761e-06,5.54690169014921e-06,5.87474559183878e-06,6.15193366115622e-06,6.39204488840697e-06,6.60383800075994e-06,6.79329339563858e-06,8.03968510112794e-06,8.76877751004911e-06,9.28607680661731e-06,9.68732531384891e-06,1.00151692155385e-05,1.02923572848559e-05,1.05324685121067e-05,1.07442616244596e-05,1.09337170193383e-05},
yticklabels={${-10^{-5}}$,${-10^{-6}}$,,,,,,,,,${-10^{-7}}$,,,,,,,,,${0}$,${10^{-7}}$,,,,,,,,,${10^{-6}}$,,,,,,,,,${10^{-5}}$},
scaled y ticks=false,
tick label style={font=\scriptsize},
},
colormap={mymap}{[1pt]
  rgb(0pt)=(0.176470588235294,0,0.294117647058824);
  rgb(1pt)=(0.329411764705882,0.152941176470588,0.533333333333333);
  rgb(2pt)=(0.501960784313725,0.450980392156863,0.674509803921569);
  rgb(3pt)=(0.698039215686274,0.670588235294118,0.823529411764706);
  rgb(4pt)=(0.847058823529412,0.854901960784314,0.92156862745098);
  rgb(5pt)=(0.968627450980392,0.968627450980392,0.968627450980392);
  rgb(6pt)=(0.996078431372549,0.87843137254902,0.713725490196078);
  rgb(7pt)=(0.992156862745098,0.72156862745098,0.388235294117647);
  rgb(8pt)=(0.87843137254902,0.509803921568627,0.0784313725490196);
  rgb(9pt)=(0.701960784313725,0.345098039215686,0.0235294117647059);
  rgb(10pt)=(0.498039215686275,0.231372549019608,0.0313725490196078)
},
point meta max=1.11254994381888e-05,
point meta min=-1.11254994381888e-05,
xmin=0,
xmax=81,
xlabel=$\vec z_u$,
ylabel=$\vec z_v$,
ymin=0,
ymax=81,
]
\addplot graphics [includegraphics cmd=\pgfimage,xmin=0, xmax=81, ymin=0, ymax=81] {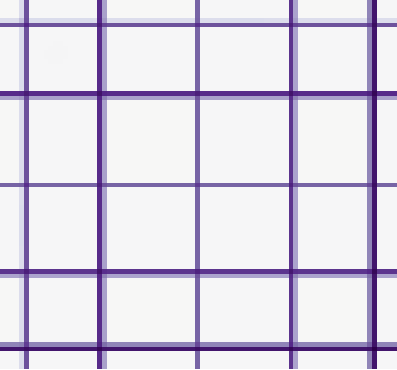};
\end{axis}

\end{tikzpicture}
        \caption{$\Sigma(\vec z_u,\vec z_v)$}\label{fig:ex2:comp}
    \end{subfigure}
    \hfill
    \begin{subfigure}[t]{0.49\textwidth}
        \centering
        \begin{tikzpicture}

\begin{axis}[
height=0.75\textwidth,
colormap/blackwhite,
point meta max=0,
point meta min=-1.11254994381888e-05,
xlabel=$\vec z_u$,
ylabel=$\vec z_v$,
xmin=0, xmax=81,
ymin=0, ymax=81,
axis on top,
major tick length=0pt
]
\addplot graphics [includegraphics cmd=\pgfimage,xmin=0, xmax=81, ymin=0, ymax=81] {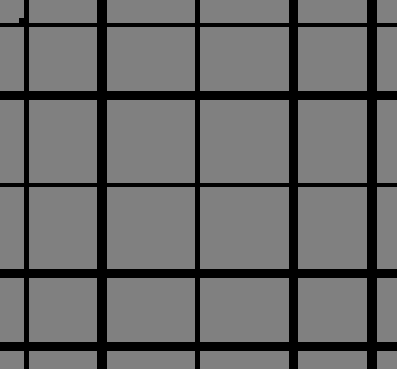};
\end{axis}

\end{tikzpicture}
        \caption{pairs marked numerically positive (white), numerically zero (gray), numerically negative (black)}\label{fig:ex2:test}
    \end{subfigure}
    \caption{Example 3: values of stationarity test and distribution of failed pairs}
    \label{fig:ex2_stat}
\end{figure}
Using the tolerance $\mathrm{tol}=1.11\cdot 10^{-7}$ leads to $0$ numerically positive, 
$5328$ numerically zero, and $1233$ numerically negative pairs, see \cref{fig:ex2:test}. 
These are more than $18.5\%$ of all tested pairs.
In this example, $\Theta=-3.35\cdot 10^{-10}$ holds true.

\paragraph{Summary}

The results of the numerical experiments are summarized in \cref{tab:summary}, where ``complementarity'' 
refers to the maximal absolute value of the elementwise Fischer--Burmeister function.
Noting that Experiment 1 provides a benchmark for a passed stationarity test,
a computed solution of \eqref{eq:OCCC} is considered as \emph{approximately passing} the strong stationarity test 
if $|\Theta|\leq\sqrt{\mathrm{tol}}$ 
holds for $\Theta$ defined in \eqref{eq:test_discrete_first_condition}
and the tolerance defined in \eqref{eq:definition_tolerance},
while the number of numerically negative tested pairs is at most $10\%$ of the total number of tested pairs.
\begin{table}[t]
    \centering
    \begin{tabular}{rccc}
        \toprule
        & Example 1 			& Example 2            	& Example 3     		\\
        \midrule
        $y_{\text{d}}$    & in $L^{2}(\Omega)$ 	& in $H^{1}(\Omega)$   	& constant         		\\
        $\varepsilon$     & $10^{-8}$				& $10^{-8}$            	& $10^{-8}$             \\
        \midrule
        complementarity   & $2.08\cdot 10^{-5}$ 	& $3.58\cdot 10^{-6}$  	& $2.02\cdot10^{-6}$  	\\
        $\text{tol}$      & $2.27\cdot 10^{-10}$	& $1.62\cdot 10^{-8}$  	& $1.11\cdot 10^{-7}$  	\\
        $\Theta$          & $-1.65\cdot 10^{-7}$	& $-2.01\cdot 10^{-9}$ 	& $-3.35\cdot 10^{-10}$	\\
        num.\ neg.\ pairs & $8.3\%$				& $4.7\%$              	& $18.5\%$              \\
        stationarity test & passed 				& passed               	& failed                \\
        \bottomrule
    \end{tabular}
    \caption{summary of experiments}
    \label{tab:summary}
\end{table}

It has to be mentioned that more experiments with the same parameter settings of the above three examples 
were implemented for unstructured grids.  
In \cref{alg:pathfollowing}, the inner iteration implements a damped Newton method to compute the optimal solution 
of the KKT system \eqref{eq:NOC_discrete} with the fixed penalty parameter $\sigma_k$, which increases in every outer loop.  
All experiments show that there is no significant correlation between the number of (inner) Newton iterations 
and the mesh size. 
However, the solutions calculated on unstructured grids differ significantly from those ones obtained on structured grids,
and this phenomenon is not restricted to the use of basis functions from $\mathcal{P}^1(\Omega_\Delta)$.
The reason behind this fact may be the inherent nonconvexity of the optimal control problem \eqref{eq:OCCC},
which causes the existence of several local minimizers (and thus strongly stationary points).
This also explains the observed fact that the output of \cref{alg:pathfollowing} heavily relies on the initial
guess for the controls.

\section{Conclusions}\label{sec:conclusions}

Optimal control problems with complementarity constraints on the controls admit solutions 
if the controls are chosen from a first-order Sobolev space. 
Although necessary optimality conditions of strong stationarity-type can be derived in this case, 
the explicit characterization of the associated Lagrange multipliers is difficult and remains the topic of further research. 
However, a penalty method based on the Fischer--Burmeister function can be formulated 
that ensures convergence to a global minimizers of the original complementarity-constrained problem. 
In theory, this requires computing global minimizers of the penalized problems, 
and it has to be investigated whether an adapted method based on KKT points is theoretically possible. 
Nevertheless, numerical examples illustrate that combined with a computable check for a discrete strong stationarity-type condition, 
this approach leads to a numerical procedure that in many cases results in nearly strongly stationary points. 
In light of prominent literature which deals
with the numerical treatment of finite-dimensional complementarity problems, see \cite{HoheiselKanzowSchwartz2013}
and the references therein, this seems to be the best to be hoped for.

\appendix

\section{A helpful lemma}

In the proof of \cref{prop:strong_convergence_of_surrogate_solutions}, the following lemma is used twice.
\begin{lemma}\label{lem:sum_of_real_sequences}
    Let $\{\alpha_k\}_{k\in\N},\{\beta_k\}_{k\in\N}\subset\R$ be sequences such that 
    $\alpha_k+\beta_k\to\alpha+\beta$ holds where $\alpha,\beta\in\R$ satisfy
    \[
        \alpha\leq\liminf_{k\to\infty}\alpha_k,\qquad\beta\leq\liminf_{k\to\infty}\beta_k.
    \]
    Then, the convergences $\alpha_k\to\alpha$ and $\beta_k\to\beta$ are valid.
\end{lemma}
\begin{proof}
    The assumptions imply that
    \begin{align*}
        \alpha
        \leq\liminf_{k\to\infty}\alpha_k
        &\leq\limsup_{k\to\infty}\alpha_k
        =\limsup_{k\to\infty}(\alpha_k+\beta_k-\beta_k)\\
        &=\lim_{k\to\infty}(\alpha_k+\beta_k)+\limsup_{k\to\infty}(-\beta_k)
        \leq \alpha+\beta -\beta =\alpha,
    \end{align*}
    which implies that $\alpha_k\to\alpha$. Now, $\beta_k\to\beta$ follows from $\alpha_k+\beta_k\to\alpha+\beta$.
\end{proof}

\section*{Acknowledgments}

The authors sincerely thank Frank Heyde for fruitful discussions about the explicit form of the 
generalized second-order derivative of the discretized squared Fischer--Burmeister function.
Furthermore, the authors appreciate the comments of two anonymous reviewers which helped to
improve the presentation of the obtained results.
This work is partially supported by the DFG grants
\emph{Parameter Identification in Models With Sharp Phase Transitions}
and
\emph{Analysis and Solution Methods for Bilevel Optimal Control Problems}
under the respective grant numbers CL\,487/2-1 and DE\,650/10-1 
within the Priority Program SPP 1962
(Non-smooth and Complementarity-based Distributed Parameter Systems: Simulation and Hierarchical Optimization).

\printbibliography

\end{document}